\documentclass[12pt]{amsart}

\usepackage{amssymb}
\usepackage[all]{xy}
\usepackage{graphics}
\usepackage{color}
\usepackage[T1]{fontenc}
\usepackage{fullpage}
\usepackage[colorlinks, linkcolor=magenta, citecolor=blue]{hyperref}
\usepackage{enumitem}

\newcommand{\B}[1]{\mathbb{#1}}
\newcommand{\C}[1]{\mathcal{#1}}
\newcommand{\F}[1]{\mathfrak{#1}}

\usepackage[colorinlistoftodos,prependcaption]{todonotes}

\newtheorem{theorem}[subsection]{Theorem}
\newtheorem{corollary}[subsection]{Corollary}
\newtheorem{lemma}[subsection]{Lemma}
\newtheorem{proposition}[subsection]{Proposition}

\theoremstyle{definition}
\newtheorem{definition}[subsection]{Definition}

\newtheorem{example}[subsection]{Example}

\theoremstyle{remark}
\newtheorem{remark}[subsection]{Remark}
\newtheorem*{remarknonum}{Remark}

\numberwithin{figure}{section}
\numberwithin{table}{section}
\numberwithin{equation}{section}

\def\conj{\operatorname{conj}}
\def\defeq{\overset{\text{def}}{=}}
\def\diam{\operatorname{diam}}
\def\PSL{\operatorname{PSL}}
\def\SL{\operatorname{SL}}
\def\ZZ{\mathbb{Z}}

\def\defeq{\overset{\text{def}}{=}}

\def\lAngle{\langle\langle}
\def\rAngle{\rangle\rangle}

\newcommand{\OP}{\operatorname}

\def\Ad{\OP{Ad}}
\def\GL{\OP{GL}}
\def\Nil{\OP{Nil}}

\def\Knapp{MR1920389}
\def\HigNeeb{MR3025417}

\newcommand{\xto}[1]{\xrightarrow{#1}}

\begin{document}

\title{Strong and uniform boundedness of groups}
\author{Jarek K\k{e}dra}
\address{Institute of Mathematics,
University of Aberdeen,
King's College,
Fraser Noble Building,
Aberdeen AB24 3UE,
United Kingdom, and University of Szczecin}
\email{kedra@abdn.ac.uk}
\author{Assaf Libman}
\address{Institute of Mathematics,
University of Aberdeen,
King's College,
Fraser Noble Building,
Aberdeen AB24 3UE,
United Kingdom}
\email{a.libman@abdn.ac.uk}
\author{Ben Martin}
\address{Institute of Mathematics,
University of Aberdeen,
King's College,
Fraser Noble Building,
Aberdeen AB24 3UE,
United Kingdom}
\email{b.martin@abdn.ac.uk}


\begin{abstract}
A group $G$ is called {\em bounded} if every conjugation-invariant norm on $G$
has finite diameter.  We introduce various strengthenings of this property and
investigate them in several classes of groups including semisimple Lie groups,
arithmetic groups and linear algebraic groups. We provide applications to
Hamiltonian dynamics. 
\end{abstract}

\maketitle


\section{Introduction and statements of results} \label{Sec:Introduction}
Conjugation-invariant norms on groups appear in various branches of
mathematics including Hamiltonian dynamics (the Hofer norm), finite groups
(covering numbers), geometric group theory (verbal norms) and others.
Bu\-ra\-go, Ivanov and Polterovich introduced the concept of a {\it bounded
group} \cite{MR2509711}: that is, a group for which every conjugation-invariant
norm has finite diameter. 

A subset $S$ of a group $G$ {\em normally generates} it if $G$ is the normal subgroup generated by $S$.
We say that $G$ is {\em finitely normally generated} if it admits a finite normally generating set $S$.
Such $S$ gives rise to a word norm $\| \cdot \|_S$ on $G$, where $\|g\|_S$ is the length of the shortest word in the conjugates of the elements of $S$ and their inverses needed to express $g$.
By construction, $\| \cdot \|_S$ is conjugation-invariant.
We will write $\|G\|_S$ for the diameter of $G$ with respect to the norm $\| \cdot \|_S$; see Section \ref{S:preliminaries} for more details.

If $G$ is finitely normally generated then being bounded is equivalent to all the word norms on $G$ having finite diameter (Corollary \ref{C:tfae bounded}). 
In light of this, the purpose of this paper is to refine the notion of boundedness of word norms and study consequences of such refinements.

\subsection*{Strong and uniform boundedness}

Let $G$ be finitely normally generated.
For any $k \geq 1$ define
\[
\Delta_k(G) = \sup \{ \| G\|_S \ : \ \text{$S$ normally generates $G$ and $|S| \leq k$}\}
\]
with the convention that $\sup \emptyset=-\infty$.
It is clear that $\Delta_1(G) \leq \Delta_2(G) \leq \dots$ and the limit of this sequence is
\[
\Delta (G) = \{ \| G\|_S \ : \ \text{$S$ is a finite normally generating set of $G$} \}.
\]

\begin{definition}
A finitely normally generated group $G$ is called {\em strongly bounded} if $\Delta_k(G)<\infty$ for all $k$.
It is called {\em uniformly bounded} if $\Delta(G)<\infty$.
\end{definition}

We remark that our definition of strong boundedness is unrelated to those of Cornulier \cite{MR2240370} and Le Roux-Mann \cite{leroux-mann}.  
Given Corollary \ref{C:tfae bounded} below, within the class of finitely normally generated groups there are inclusions
\[
\{\text{uniformly bounded}\} \subseteq \{\text{strongly bounded}\} \subseteq
\{\text{bounded}\}.
\]
One goal of the paper is to give examples and study the properties of groups in these classes.
The next theorem shows that simple Lie groups with finite center provide examples of groups at the two extremes.

\begin{theorem}[Theorem \ref{T:Lie groups boundedness properties}]
\label{T:Intro Lie groups}
Let $G$ be a semisimple Lie group.  
Then  $G$ is finitely normally generated, and
\begin{enumerate}[label=(\alph*)]
\item $G$ is bounded if and only if  $Z(G)$ is finite. 
\end{enumerate}
If $Z(G)$ is finite then the following hold. 
\begin{enumerate}[label=(\alph*)]
\setcounter{enumi}{1}
\item 
If $G/Z(G)$ has a non-trivial compact factor then $G$ is bounded but not strongly bounded. 

\item 
If $G/Z(G)$ has no non-trivial compact factors then $G$ is uniformly bounded. 
\end{enumerate}
\end{theorem}



In the non-compact case it is possible to find an explicit upper bound for
$\Delta(G)$ which only depends on $\OP{rank} G$. We will do this in a forthcoming
paper.  The next result provides another family of uniformly bounded groups.

\begin{theorem}[Proposition \ref{P:alggp} and Theorem \ref{T:alggp_unif_bd}]
Every linear algebraic group with finite abelianization over an algebraically
closed field is uniformly boun\-ded.  \end{theorem}

Finding strongly bounded groups that are not uniformly bounded is a more difficult challenge.
Application of Corollary \ref{C:Carter Keller implications} to $\C O = \B Z$ gives:

\begin{theorem}\label{T:Intro slnZ}
For any $n \geq 3$ the group $\OP{SL}(n,\B Z)$ is strongly bounded but not uniformly bounded.
\end{theorem}

\begin{remarknonum}
In contrast, $\OP{SL}(n,\C R)$ is uniformly bounded, where $n \geq 3$ and $\C R$ is a principal ideal domain with only finitely many maximal ideals (Theorem \ref{T:slnr main theorem:2}).
\end{remarknonum}

\begin{remarknonum}

Theorem \ref{T:Intro slnZ} is related to results of \cite{MR2357719}; indeed, one can show using Corollary 3.8 and Proposition 6.7 of \cite{MR2357719} that $\OP{SL}(n,\C R)$ is strongly bounded for a large class of rings $\C R$, including $\C R= \B Z$ (we thank Dave Morris for this observation).  This argument uses the Compactness Theorem from first-order logic, and it does not yield any explicit bound for $\Delta_k(\OP{SL}(n,\C R))$.  Remark (6.2) of \cite{MR2357719} is incorrect since it implies that
$\Delta(\OP{SL}(n,\B Z))$ is bounded by a function of $n \geq 3$, which
contradicts Theorem \ref{T:Intro slnZ}.

\end{remarknonum}

Before passing to applications, we mention that 
uniform boundedness imposes group-theoretic restrictions.

\begin{theorem}[Theorem \ref{T:simple_quotient}]
A uniformly bounded group has only finitely many maximal normal subgroups.
\end{theorem}

\noindent For an application to linear groups, see Theorem~\ref{T:strong-approximation}, which
states that finitely generated Zariski dense subgroups of certain algebraic groups
are not uniformly bounded.

\subsection*{An application to cocompact lattices}

It is an open problem whe\-ther finitely generated cocompact lattices in
semisimple Lie groups are boun\-ded. Many such lattices can be embedded as dense
subgroups in compact simple Lie groups.  For example, 
$\OP{SO}(n,\B Z[1/5])\subset \OP{SO}(n)$ for $n\geq 5$ is such a group
\cite[Example 3.2.2 (B), Example 3.2.4 (B) and Proposition 3.2.2]{MR1308046}.
Our next result, which immediately follows from Proposition~\ref{P:dense subgroups}\ref{I:dense subgroups:dense not strongly bounded},
implies that such lattices are not strongly bounded.

\begin{theorem}\label{thm:dense_Lie}
Let $G$ be a compact simple Lie group and let $H$ be a finitely normally generated group.
If $H\to G$ is a homomorphism with dense image then $H$ is not strongly bounded.
\end{theorem}

\subsection*{Applications to finite groups of Lie type}

Clearly any finite group is uniformly boun\-ded.
The value of $\Delta(G)$ is related to the size of the conjugacy classes.
We prove the following results in Section \ref{S:applications to finite groups}.

\begin{proposition}[Example \ref{E:PSL(n,q) ccs}; compare with {\cite[Corollary 4.3]{MR3198721}}]
Let $n \geq 3$ and $q$ a prime power.
Then $\Delta(\PSL(n,q)) \leq 12(n-1)$ and consequently, if $S$ is a non-trivial conjugacy class then 
\[
\log |S| > \frac{\log |G|}{\Delta(G)}- 2 \geq (n+1) \cdot \frac{\log q}{12}- \frac{\log q+ 2}{12(n-1)}- 2.
\]
\end{proposition}

\begin{proposition}\label{P:slnzl ccs}
Let $\ell$ be an integer and $p_1,\dots,p_k$ its distinct prime factors.
Let $n \geq 3$.
Then
\[
\Delta(\OP{SL}(n,\B Z/\ell)) \leq 12k(n-1)
\]
and if $S$ is the conjugacy class of a matrix $A \in \OP{SL}(n,\B Z/\ell)$ whose reduction modulo $p_i$ is not scalar in $\OP{SL}(n,\B Z/p_i)$ for all $i$ then 
\[
\log |S| \geq \frac{\log|\SL(n,\ZZ/\ell)|}{12k(n-1)}- 2.
\]
\end{proposition}

\subsection*{Applications to Hamiltonian dynamics}

Let $(M,\omega)$ be a closed (i.e., compact without boundary)
symplectic manifold and let~$\OP{Ham}(M,\omega)$ denote the group of Hamiltonian
diffeomorphisms of $(M,\omega)$.  
This group is simple \cite[Theorem 4.3.1.(ii)]{MR1445290}.
For background on symplectic manifolds and
Hamiltonian actions see, for example,  Arnold-Khesin \cite{MR1612569} or
McDuff-Salamon \cite{MR2000g:53098}.  


The following theorem gives information about the subgroup structure of
$\OP{Ham}(M,\omega)$.  
Part \eqref{I:non embedding Diff} is an immediate consequence of 
\cite[Theorem 1.11(i)]{MR2509711}.
Part \eqref{I:non embedding in ham Lie groups} is
related to a result of Delzant \cite{delzant} which says that a non-compact
simple Lie group $G$ cannot act {\em smoothly} on $M$. 
Another proof is due to Polterovich and Rosen \cite[Proposition 1.3.18]{MR3241729}, again for smooth actions.  
Our argument works for all actions, not just smooth ones.

\begin{theorem}
No subgroup of $\OP{Ham}(M,\omega)$ is abstractly isomorphic to any one of the following groups.
\begin{enumerate}
\item A semisimple Lie group $G$ with finite center and no non-trivial compact factors. \label{I:non embedding in ham Lie groups}
\item A semisimple algebraic group $G$ over an uncountable algebraically closed field. \label{I:non embedding in ham algebraic groups}
\item The automorphism group $G$ of a regular tree with vertices of valence at least $3$. \label{I:non embedding in ham trees}
\item The identity component $\OP{Diff}_0(N)$ of the group $\OP{Diff}(N)$ of compactly supported
diffeomophisms of a connected smooth manifold $N$. \label{I:non embedding Diff}
\end{enumerate}
\end{theorem}
\begin{proof}
The Hofer norm \cite[Section 12.3]{MR2000g:53098} is a nondiscrete conjugation-invariant norm on $H=\OP{Ham}(M,\omega)$.
The identity map from the $C^1$-topology to the Hofer topology is continuous \cite[Proposition 5.10]{MR96h:58063}, and since the $C^1$-topology is separable \cite[Section 2]{MR1336822}, the Hofer topology is separable too.

(\ref{I:non embedding in ham Lie groups}) 
Suppose that $G$ is a semisimple Lie group with no compact factors and finite center.
Then $G/Z(G)$ is a product of simple non-compact centre-free Lie groups (see Section \ref{S:Lie groups}).
Since centre-free simple Lie groups are simple abstract groups  (see \cite[Proposition 6.30]{\Knapp} and Lemma \ref{L:bg2n contains nbhd}), it follows from Theorem \ref{T:Intro Lie groups} that~ $G$ has a finite composition series with all factors uniformly bounded.
Then~$G$ is not isomorphic (abstractly) to a subgroup of $H$ by Corollary~\ref{C:uniformly simple quotients}\ref{I:uniformly simple quotients:non-embedding}.

(\ref{I:non embedding in ham algebraic groups})
A semisimple algebraic group $G$ over an algebraically closed field $k$ admits a normal series such that each factor group is a simple algebraic group.  If $H$ is a simple algebraic group over $k$ then $|Z(H)|<\infty$  and $H/Z(H)$ is simple as an abstract group \cite[Section 27.5 and Corollary 29.5]{MR0396773}.  Hence $G$ has a composition series such that each composition factor is either of the form $H/Z(H)$ or a finite simple group.
The latter are clearly uniformly bounded, and it follows from Proposition \ref{P:alggp} and  Theorem \ref{T:alggp_unif_bd} that each $H/Z(H)$ is uniformly bounded.
Clearly $G$ is uncountable and it follows from Corollary \ref{C:uniformly simple quotients}\ref{I:uniformly simple quotients:non-embedding} that $G$ is not isomorphic to a subgroup of $H$.

(\ref{I:non embedding in ham trees}) Let $T$ be such a tree.  Then
$G=\OP{Aut}(T)$  is uncountable because it acts transitively on the boundary of
$T$ which is a Cantor set.  It follows from \cite[Theorem 3.4]{MR3085032} that
$G$ is simple and uniformly bounded and we can apply Corollary \ref{C:uniformly
simple quotients}\ref{I:uniformly simple quotients:non-embedding} to show that
$G$ is not isomorphic to any subgroup of $H$.

(\ref{I:non embedding Diff})
It is shown in \cite[Theorem 1.11(i)]{MR2509711} that every conjugation-invariant norm on $G=\OP{Diff}_0(N)$ is discrete and $G$ is clearly uncountable.
Hence, this group cannot be a subgroup of $\OP{Ham}(M,\omega)$.
\end{proof}

\subsection*{Acknowledgements} We thank Philip Dowerk, \'Swiatos\l aw Gal,
\'Etienne Ghys, Vincent Humili\`ere, Morimichi Kawasaki, Nicolas Monod, Dave
Morris, Leonid Polterovich and Yehuda Shalom for helpful comments and for
answering our questions.  This work was funded by Leverhulme Trust Research
Project Grant RPG-2017-159.

%
%
%
\section{Norms and boundedness}\label{S:preliminaries}

In this section we introduce the central concepts of this paper: {\em strong boundedness} and {\em uniform boundedness}.

\subsection*{Conjugation-invariant norms}
Let $G$ be a group.
A {\em norm} on $G$ is a non-negative valued function $\nu \colon G \to \B R$ such that
\begin{itemize}
\item[(a)] $\nu(g)=0 \iff g=1$.
\item[(b)] $\nu(g^{-1})=\nu(g)$ for all $g \in G$. 
\item[(c)] $\nu(gh) \leq \nu(g)+\nu(h)$ for all $g,h \in G$.
\end{itemize}
We call $\nu$ \emph{conjugation-invariant} or \emph{bi-invariant} if in addition
\begin{itemize}
\item[(d)] $\nu(ghg^{-1})=\nu(h)$ for all $g,h \in G$.
\end{itemize}
A conjugation-invariant norm $\nu$ gives rise to a metric $d(x,y)=\nu(xy^{-1})$ invariant under left and right translation; the converse is also true.  It is easily checked that $d(x^{-1}, y^{-1})= d(x,y)$ and $d(x_1y_1, x_2y_2)\leq d(x_1,x_2)+ d(y_1,y_2)$ for all $x,y,x_1,x_2, y_1,y_2\in G$, so $G$ together with the topology induced by $d$ is a topological group.

We say that $\nu$ is {\em discrete} if it induces the discrete metric on $G$, i.e., if $\inf \{ \nu(g) : 1 \neq g \in G \} >0$.
The following result is elementary and is left to the reader.

\begin{lemma}\label{L:quotient of conj inv norms}
Let $\pi \colon G \to H$ be a group epimorphism and $\nu$ a conjugation-invariant norm on $G$.
Assume that the restriction of $\nu$ to $\ker \pi$ is discrete. 
Then the function $\nu' \colon H \to \B R$ defined by
\[
\nu'(h)=\inf\,\{ \nu(g) : g \in \pi^{-1}(h)\}
\]
is a conjugation-invariant norm on $H$.
\end{lemma}

\begin{definition}
A group $G$ is called {\em bounded} if the diameter of every conjugation-invariant norm on $G$ is finite.
\end{definition}

This concept has been studied by Burago, Ivanov and Polterovich in
\cite{MR2509711} and is, in some sense, the starting point of our
investigation. Dowerk and Thom studied boundedness properties of
the projective unitary group $\OP{PU}(M)$, where $M$ is a von Neumann
factor \cite{MR3920346,MR3907832}.

\subsection*{Word norms}
Let $X$ be a subset of a group $G$.
Let $\conj_H(X^{\pm 1})$ be the set of all $g \in G$ that are conjugate by an element of $H \leq G$ to some element of $X$ or its inverse.
Define for any $g \in G$
\[
\| g \|_X \defeq \inf \, \{ n : \text{$g=y_1\cdots y_n$ for some $y_1,\dots,y_n \in \conj_G(X^{\pm 1})$} \}
\]
Notice that $\|g\|_X=\infty$ if $g \notin \lAngle X \rAngle$, the normal subgroup generated by $X$.
For any $n \geq 0$ define
\[
B_X(n) \defeq \{ g \in G \ : \ \| g \|_X \leq n \}. 
\]
If we want to make it clear what the ambient group is we will sometimes write $B_X^G(n)$.
It is clear that $\{1\} = B_X(0) \subseteq B_X(1) \subseteq B_X(2) \subseteq \dots$ and that $\bigcup_{n \geq 0} B_X(n) = \lAngle X \rAngle$.
The following result is elementary.

\begin{lemma}\label{L:BXn properties}
Let $G$ be a group, let $X,Y \subseteq G$ and let $n,m \in \B N$.
Then
\begin{enumerate}[label=(\roman*)]
\item
\label{L:BXn:inv}
$B_X(n)^{-1}=B_X(n)$ and $B_X(n)$ is invariant under conjugation in $G$.

\item
\label{L:BXn:incl}
$X \subseteq Y \implies B_X(n) \subseteq B_Y(n)$.

\item
\label{L:BXn:mult}
$B_X(n) B_X(m) = B_X(n+m)$.

\item
\label{L:BXn:assoc}
$Y \subseteq B_X(n) \implies B_Y(m) \subseteq B_X(mn)$.

\item
\label{L:BXn:quotient}
If $\pi \colon G \to H$ is an epimorphism then $B^{H}_{\pi(X)}(n)=\pi(B^G_X(n))$ for any $X \subseteq G$.

\item
\label{L:BXn:quotient 2}
If $\pi \colon G \to H$ is an epimorphism then $B^{G}_{\pi^{-1}(Y)}(n)=\pi^{-1}(B^H_Y(n))$ for any $Y \subseteq H$.

\end{enumerate}
\end{lemma}

If $X \subset G$ {\em normally generates} $G$, i.e., $\lAngle X \rAngle =G$, then $g \mapsto \|g\|_X$ is a conjugation-invariant norm on $G$.
We will write $\| G\|_X$ for the diameter of the norm $\| \cdot \|_X$.
If $X=\{s\}$ is a singleton, we will write $\| g \|_s$ instead of $\| g\|_{\{s\}}$ and likewise $\| G \|_s$ instead of $\| G \|_{\{s\}}$.

\begin{proposition}\label{P:lipschitz-max}
Let $G$ be a group normally generated by a finite set $S$.
Let $\psi \colon G \to H$ be a homomorphism and $\nu$ a conjugation-invariant norm on $H$.
Then $\psi$ is Lipschitz with constant $C=\max \, \{ \nu(\psi(s)) : s \in S\}$: that is,
\[
\nu(\psi(g)) \leq C \| g\|_S \qquad \text{ for any $g \in G$.}
\]
\end{proposition}

\begin{proof}
Any $g \in G$ has the form $g=x_1\cdots x_n$ where $n=\|g\|_S$ and each $x_i$ is conjugate to some $s_i \in S$ or its inverse.
Since $\nu$ is a conjugation-invariant norm, $\nu(\psi(g)) \leq \sum_{i=1}^n \nu(\psi(x_i)) = \sum_{i=1}^n \nu(\psi(s_i)) \leq Cn=C\|g\|_S$.
\end{proof}

Call $G$ {\em finitely normally generated} if it is normally generated by a finite $S \subseteq G$.
Set
\begin{eqnarray*}
\Gamma_n(G) &=& \{ S \subseteq G \ :\ \text{$|S|\leq n$ and $S$ normally generates $G$}\}, \\
\Gamma(G) &=& \{ S \subseteq G \ :\ \text{$|S|<\infty$ and $S$ normally generates $G$}\}.
\end{eqnarray*}

In finitely normally generated groups, boundedness is determined by the behaviour of word norms $\| \cdot \|_X$.

\begin{corollary}\label{C:tfae bounded}
Let $G$ be a finitely normally generated group.
The following conditions are equivalent.
\begin{enumerate}
\item $G$ is bounded.
\item $\| G \|_S <\infty$ for some $S\in \Gamma(G)$.
\item $\|G\|_S<\infty$ for every $S\in \Gamma(G)$.
\qed
\end{enumerate}
\end{corollary}

\begin{proof}
Clearly (1) $\implies$ (3) since $\| \cdot \|_S$ is a conjugation-invariant norm, and (3) $\implies$ (2) is trivial since $\Gamma(G) \neq \emptyset$.
To prove (2) $\implies$ (1) apply Proposition~\ref{P:lipschitz-max} to $\OP{id} \colon G \to G$.
\end{proof}

\subsection*{Strong and uniform boundedness}
In light of Corollary \ref{C:tfae bounded} we refine the notion of boundedness.

\begin{definition}\label{D:strongly and uniformly bounded}
Let $G$ be a finitely normally generated group.
Set
\begin{eqnarray*}
&& \Delta(G) = \OP{sup} \{ \OP{diam}(\nu_S) \ : \ S \in \Gamma(G)\} \\
&& \Delta_n(G) = \OP{sup} \{ \OP{diam}(\nu_S) \ : \ S \in \Gamma_n(G)\}
\end{eqnarray*}
where $\Delta_n(G)=-\infty$ if $\Gamma_n(G)=\emptyset$.
We say that $G$ is \emph{strongly bounded} if $\Delta_n(G)<\infty$ for all $n$.
We say that $G$ is \emph{uniformly bounded} if $\Delta(G) < \infty$.
\end{definition}

Clearly, $\Delta_1(G) \leq \Delta_2(G) \leq \dots$ and 
\[
\Delta(G)= \sup_{n \geq 1} \Delta_n(G) = \lim_{n \to \infty} \Delta_n(G).
\]

\begin{example}\label{E:deltas of simple groups}
Let $G$ be a (non-trivial) simple group.
Then $G$ is normally generated by any non-identity element.
Any $S \in \Gamma(G)$ must contain some $1 \neq x \in S$ and $\|g\|_S \leq \| g \|_x \leq \|G\|_x \leq  \Delta_1(G)$ for any $g \in G$, hence $\|G\|_S \leq \Delta_1(G)$. 
Since $S$ was arbitrary,
\[
\Delta(G)=\Delta_1(G).
\]
\end{example}

\subsection*{Subgroups, quotient, extensions}
Strong and uniform boundedness don't behave well with respect to subgroups.

\begin{example}
Uniformly bounded groups may contain unbounded normal subgroups of finite index.
An example is the inclusion of $\B Z$, which is clearly unbounded, in the infinite dihedral group $G=\B Z/2 \ltimes \B Z$.
To see that $G$ is uniformly bounded, let $N=2 \B Z$ and $K=\B Z/2 \ltimes 2 \B Z$.
The conjugacy class of any $y \notin \B Z$ is the coset $yN$ and therefore $N \subseteq B_y(2)$.
Then $G$ is finitely normally generated since $[G:N]=4$.
If $S \in \Gamma(G)$ then its image $T$ in $G/N \cong \B Z/2 \times \B Z/2$ normally generates it and clearly $\|G/N\|_T\leq 2$.
By Lemma \ref{L:BXn properties}\ref{L:BXn:quotient} the image of $B_S(2)$ in $G/N$ is $B_T(2)$, hence $G=B_S(2) \cdot N \subseteq B_S(2)B_S(2)=B_S(4)$.
Since $S$ was arbitrary, $\Delta(G) \leq 4$.  (In fact, it is not hard to show that $\Delta(G)= 3$.)
\end{example}

Quotients of strongly and uniformly bounded groups are better behaved.

\begin{lemma}\label{L:finite gen quotient not strongly bounded}
Suppose that $G$ is normally generated by $n$ elements.
Let $\pi \colon G \to H$ be an epimorphism.
\begin{enumerate}[label=(\alph*)]
\item 
\label{I:finite gen quotient not strongly bounded:1}
If $G$ is bounded, then $H$ is bounded.
\item
\label{I:finite gen quotient not strongly bounded:2}
$\Delta_k(H) \leq \Delta_{n+k}(G)$ for all $k \geq 1$.
In particular, if $G$ is strongly (resp. uniformly) bounded then so is $H$.
\end{enumerate}
\end{lemma}

\begin{proof}
Let $Y=\{y_1,\dots,y_n\} \in \Gamma_n(G)$. 

{\em Claim:}
Let $X \in \Gamma_k(H)$.
Then  $\|H\|_X \leq \| G\|_Z$ for some $Z \in \Gamma_{n+k}(G)$.

{\em Proof:} Choose a set-theoretic section $\sigma \colon H \to G$.
Since $Y$ is finite, $\pi(Y) \subseteq B_X^H(r)$ for some $r \geq 1$.
By Lemma \ref{L:BXn properties}\ref{L:BXn:quotient}, $\pi(B_{\sigma(X)}^G(r)) =B_X^H(r)$.
Therefore there are $w_1,\dots,w_n \in B_{\sigma(X)}^G(r)$ such that $\pi(w_i)=\pi(y_i)$. 
Set
\[
Z=\sigma(X) \cup \{ y_iw_i^{-1} \ : \ 1 \leq i \leq n\}.
\]
By Lemma \ref{L:BXn properties}\ref{L:BXn:mult}, $y_i = y_iw_i^{-1} \cdot w_i \in B_Z(1+r)$, and therefore $Z$ normally generates $G$ because $Y$ does.
Thus, $Z \in \Gamma_{n+k}(G)$ and for any $g \in G$,
\[
\| \pi(g) \|_X = \| \pi(g) \|_{X \cup \{1\}}  =  \| \pi(g) \|_{\pi(Z)} \leq \| g \|_Z \leq \|G \|_Z.
\]
Since $\pi$ is surjective, $\| H \|_X \leq \| G\|_Z$, which proves the claim.

Clearly, $\pi(Y)$ is a finite normal generating set for $H$; \ref{I:finite gen quotient not strongly bounded:1} follows from the claim and  Corollary \ref{C:tfae bounded}.
To prove \ref{I:finite gen quotient not strongly bounded:2}, let $X \in \Gamma_k(H)$.
The claim implies that $\|H\|_X \leq \| G \|_Z \leq \Delta_{k+n}(G)$.
The inequality follows by taking the supremum over all $X$.
\end{proof}

Extensions of strongly and uniformly bounded groups behave well under some finiteness assumptions.

\begin{lemma}
\label{L:finite_extn new}
Let $H$ be a finitely normally generated group.
Let $G\xto{\pi} H$ be a group epimorphism with finite kernel $N$ of order $n$.
Then $G$ is finitely normally generated, and
\begin{enumerate}[label=(\alph*)]
\item
\label{I:finite_extn new:1}
If $H$ is bounded then $G$ is bounded.
\item
\label{I:finite_extn new:2}
For any $k \geq 1$,
\[
\Delta_k(G) \leq (2n-1)\Delta_k(H)+n-1.
\]
\end{enumerate}
\noindent Hence, if $H$ is strongly (resp., uniformly) bounded then so is $G$.
\end{lemma}

\begin{proof}
Choose $S \in \Gamma_k(H)$ for some $k\geq 1$.
It follows from Lemma \ref{L:BXn properties}\ref{L:BXn:quotient 2} that $\pi^{-1}(S) \in \Gamma_{kn}(G)$ and in particular $G$ is finitely normally generated.
Lemma \ref{L:BXn properties}\ref{L:BXn:quotient} also shows that $\|g\|_{\pi^{-1}(S)} \leq \|\pi(g)\|_S \leq \| H\|_S$ for any $g \in G$, so $\|G\|_{\pi^{-1}(S)} \leq \|H\|_S$.
Item \ref{I:finite_extn new:1} now follows from Corollary \ref{C:tfae bounded}.


\ref{I:finite_extn new:2}
If $\Gamma_k(G)$ is empty then the inequality is trivial so we assume otherwise.
Choose some $S \in \Gamma_k(G)$.
Then $\pi(S) \in \Gamma_k(H)$ and we set $d=\| H\|_{\pi(S)}$.
We claim that for any $m \geq 0$,
\begin{equation}\label{E:finite_extn new:induction}
B_S^G(m+2d+1)=G \text{ \ or \ } N \cap B^G_S(m+2d+1) \supsetneq N \cap B^G_S(m).
\end{equation}
Assume that $B_S^G(m+2d+1)\subsetneq G$.
Then $B_S^G(m+d) \subsetneq G$, and since $S$ normally generates $G$, it follows that $B_S^G(m+d) \subsetneq B_S^G(m+d+1)$.
Let $g\in B_S^G(m+d+1)$ such that $g\not\in B_S^G(m+d)$.
Since $\pi(B^G_S(d))=H$ by Lemma \ref{L:BXn properties}\ref{L:BXn:quotient}, there exist $\tilde{g} \in B^G_S(d)$ such that $g \tilde{g}^{-1} \in N$.
We claim that $g \tilde{g}^{-1} \notin B^G_S(m)$: for otherwise $g=g\tilde{g}^{-1} \cdot \tilde{g} \in B^G_S(m+d)$ by Lemma \ref{L:BXn properties}\ref{L:BXn:mult}, which contradicts the choice of $g$.
In addition, $g\tilde{g}^{-1} \in B^G_S((m+d+1)+d)$ by Lemma \ref{L:BXn properties}\ref{L:BXn:mult}, so $g$ belongs to $N \cap B^G_S(m+2d+1)$ but not to $N \cap B^G_S(m)$.  This proves \eqref{E:finite_extn new:induction}.

Since $B_S^G(0)=\{1\} \subseteq N$ and since $|N|=n$, repeated application of \eqref{E:finite_extn new:induction} shows that $B_S^G((n-1)(2d+1)) \supseteq N$.
Since $\pi(B_S^G(d))=H$ we deduce from Lemma \ref{L:BXn properties}\ref{L:BXn:mult} that $G=N \cdot B_S^G(d) \subseteq B_S^G((n-1)(2d+1)+d)$.
Since $d = \| H \|_{\pi(S)} \leq \Delta_k(H)$,
\[
\|G \|_S \leq d+(n-1)(2d+1) \leq \Delta_k(H) (2n-1)+n-1.
\]
The inequality in \ref{I:finite_extn new:2} follows.
\end{proof}

\begin{lemma}\label{L:product of simple uniformly bounded}
Let $G_1,\dots,G_n$ be finitely normally generated groups.
Then $G=G_1 \times \dots \times G_n$ is finitely normally generated and
\begin{enumerate}[label=(\alph*)]
\item
If the groups $G_i$ are  bounded then $G$  bounded.
\label{I:fcg bounded for product}
\item
$\Delta_{k_1+\dots+k_n}(G) \geq \sum_i \Delta_{k_i}(G_i)$.
\label{I:product delta geq sum delta}
\item
If the groups $G_i$ are simple and uniformly bounded then $G$ is uniformly bounded.
\label{I:u bounded for product of simple}
\end{enumerate}
\end{lemma}

\begin{proof}
Identify the groups $G_i$ with subgroups of $G$ as the standard factors.
If $S_i \in \Gamma_{k_i}(G_i)$ it is clear that $S:=\bigcup_{i=1}^n S_i \in \Gamma_{k_1+\dots+k_n}(G)$ 
and that $\| G\|_S = \sum_{i=1}^n \|G_i\|_{S_i}$ which implies \ref{I:product delta geq sum delta}.
Corollary \ref{C:tfae bounded} implies \ref{I:fcg bounded for product}.

Assume that the groups $G_i$ are simple and $\Delta(G_i)<\infty$.
From Lemma \ref{L:finite_extn new}\ref{I:finite_extn new:2} we may assume that every $G_i$ is infinite, hence simple non-abelian.
Let $\pi_i \colon G \to G_i$ denote the projections. 
Let $S \in \Gamma(G)$.
For any $1 \leq i \leq n$ there must exist $s_i \in S$ such that $\pi_i(s_i) \neq 1$.
Since $G_i$ is not abelian, there exists $y_i \in G_i$ such that the commutator $x_i=[\pi_i(s_i),y_i]$ is not trivial, hence normally generates $G_i$. 
As an element of $G$, $x_i$ is a commutator $[s_i,y_i]$ so $\|x_i\|_S \leq 2$, and Lemma \ref{L:BXn properties}\ref{L:BXn:assoc} implies that
\[
\|G\|_S \leq 2 \|G\|_{\{x_1,\dots,x_n\}} \leq 2 \sum_i \|G_i\|_{x_i} \leq 2 \sum_i \Delta(G_i).
\]
Part \ref{I:u bounded for product of simple} follows.
\end{proof}

\subsection*{Geometric consequences}

\begin{lemma}[Nonsqueezing]\label{L:non-squeezing} 
Let $G$ be a finitely normally generated group such that $\Delta_n(G)<\infty$ for some $n \geq 1$.
Then for any conjugation-invariant norm $\nu$ on $G$,
\[
\Delta_n(G) \cdot \inf_{S\in \Gamma_n(G)} \left(\max_{s \in S}\, \nu(s)\right)
\geq 
\OP{diam}(\nu).
\] 
\end{lemma} 

\begin{proof} 
Choose some $S\in \Gamma_n(S)$.  
Proposition \ref{P:lipschitz-max} applied to the identity function on $G$ shows that 
\[
\diam \, \nu \leq \max\,\{ \nu(s) :s \in S\} \cdot \|G\|_S \leq \max\,\{ \nu(s) :s \in S\} \cdot \Delta_n(G).
\]
The result follow by taking infimum over all $S \in \Gamma_n(G)$.
\end{proof}

If $\nu$ is a norm on $G$, let $B_\nu^G(\epsilon)$ (or simply $B_\nu(\epsilon)$) denote the open $\nu$-ball of radius $\epsilon$ centred at $1 \in G$.

\begin{corollary}\label{C:uniformly-simple}
Let $G$ be a (non-trivial) uniformly bounded simple group. 
Then every conjugation-invariant norm $\nu$ on $G$ induces the discrete topology.
\end{corollary}

\begin{proof}
It follows from Example \ref{E:deltas of simple groups} and Lemma \ref{L:non-squeezing} that
\[
\inf_{1 \neq g \in G} \, \nu(g) 
=
\inf_{S \in \Gamma_1(G)} \, \left(\max_{g \in S} \nu(g)\right)
\geq 
\frac{\OP{diam}\, \nu}{\Delta_1(G)} 
=
\frac{\OP{diam}\, \nu}{\Delta(G)} 
>0.
\]
Therefore $\nu$ is discrete.
\end{proof}

\begin{corollary}\label{C:uniformly simple quotients}
Let $G$ be a group with a composition series $1=N_0 \lhd N_1 \lhd \cdots \lhd N_k=G$ such that the groups $N_i/N_{i-1}$ are uniformly bounded and simple.
\begin{enumerate}[label=(\roman*)]
\item
\label{I:uniformly simple quotients:discrete}
Any conjugation-invariant norm on $G$ induces the discrete topology.
\item
\label{I:uniformly simple quotients:non-embedding} 
If $G$ is in addition uncountable then it cannot be isomorphic to any subgroup of a group $H$ that can be equipped with a conjugation-invariant norm $\nu$  making it a separable metric space.
\end{enumerate}
\end{corollary}

\begin{proof}
\ref{I:uniformly simple quotients:discrete}
We use induction on $k$.
The base case $k=0$ is trivial and we prove the induction to $k+1$.
Assume false: so suppose $\nu$ is a non-discrete conjugation-invariant norm on $G$.
By the induction hypothesis $\nu|_{N_k}$ is discrete so  $N_k \cap B_\nu^G(\epsilon)=\{1\}$ for some $\epsilon>0$.
Lemma \ref{L:quotient of conj inv norms} shows that $G/N_k$ is equipped with a conjugation-invariant norm $\nu'$; the definition of $\nu'$ and the triangle inequality imply that $B^G_\nu(\epsilon/2)$ maps isometrically onto $B_{\nu'}^{G/N_k}(\epsilon/2)$.
In particular $\nu'$ is not discrete, contradicting Corollary \ref{C:uniformly-simple}.

\ref{I:uniformly simple quotients:non-embedding}
Suppose $G$ is a subgroup of $H$.
It follows from \ref{I:uniformly simple quotients:discrete} that  $\nu|_G$ is discrete.
This is impossible since an uncountable subset of a separable metric space cannot be discrete.
\end{proof}

\begin{proposition}\label{P:dense subgroups}
Let $N$ be a {\em proper} normal subgroup of $G$ and suppose that any $g \notin N$ normally generates $G$.
Then 
\begin{enumerate}[label=(\roman*)]
\item
\label{I:dense subgroups:Delta_1=Delta}
$\Delta_1(G)=\Delta(G).$
\end{enumerate}
Suppose that, in addition, $G$ is equipped with a
conjugation-invariant norm $\nu$ such that $N$ is closed but not open.
Then
\begin{enumerate}[label=(\roman*)]
\setcounter{enumi}{1}
\item
\label{I:dense subgroups:Delta_1=infty}
$\Delta_1(G)=\infty$.
\item
\label{I:dense subgroups:dense not strongly bounded}
Let $H$ be a finitely normally generated group and $H \to G$ a homomorphism with dense image.
Then $H$ is not strongly bounded.
\end{enumerate}
\end{proposition}

\begin{proof}
\ref{I:dense subgroups:Delta_1=Delta}
Any $S \in \Gamma(G)$ must contain some $g \notin N$ so $\| G\|_S \leq \| G\|_g \leq \Delta_1(G)$.
Since $S$ was arbitrary, $\Delta(G) \leq \Delta_1(G)$ and equality must hold.

\ref{I:dense subgroups:dense not strongly bounded}
First, by Lemma \ref{L:finite gen quotient not strongly bounded}\ref{I:finite gen quotient not strongly bounded:2} we may replace $H$ with its image in $G$ and hence assume that $H \leq G$.
Assume false, i.e., $H$ is strongly bounded.
Fix $S=\{h_1,\dots,h_m\} \in \Gamma(H)$.
Let $\epsilon >0$.
Since $H$ is dense and $N$ is closed and not open, $H \cap B_\nu^G(\epsilon)$ cannot be contained in $N$, whence we choose $k \in H \setminus N$ with $\nu(k)<\epsilon$.
Let $K \leq H$ be the normal subgroup $k$ generates in $H$.
Since $K \triangleleft H$, the closure $\overline{K}$ is normalised by $H$.
Since $H$ is dense and $\overline{K}$ is closed, $\overline{K} \trianglelefteq G$ and therefore $\overline{K}=G$ (since $k \notin N$).
In other words, $K$ is dense in $G$, and therefore so are the cosets $h_1K,\dots,h_mK$, hence we can choose $h_i' \in h_iK \subseteq H$ such that $\nu(h_i')<\epsilon$.
It is clear that
\[
X_\epsilon =\{ k,h_1',\dots,h_m'\}
\]
normally generates $H$ and that $X_\epsilon \subseteq B_\nu^G(\epsilon)$.
Since $\epsilon>0$ was arbitrary and $\Delta_{m+1}(H)<\infty$, the left-hand side of the inequality in the nonsqueezing Lemma \ref{L:non-squeezing}  vanishes, which is absurd since $\diam(\nu)>0$.

\ref{I:dense subgroups:Delta_1=infty}
First, $G$ is finitely normally generated.
Apply part \ref{I:dense subgroups:dense not strongly bounded} to $H=G$ to deduce that $G$ is not strongly bounded, hence not uniformly bounded, i.e., $\Delta(G)=\infty$.
The result follows from \ref{I:dense subgroups:Delta_1=Delta}.
\end{proof}

%
%
%
%
%

\section{Lie groups}
\label{S:Lie groups}

Throughout this section, unless otherwise stated,  all Lie algebras are defined over the real numbers.  
A connected Lie group $G$ is called {\em simple} if its Lie algebra $\F g$ is \emph{simple}, i.e., it is not abelian and has no non-trivial ideals.
It is called {\em semisimple} if $\F g$ is \emph{semisimple}, i.e., is a direct sum of simple Lie algebras. The purpose of this section is to prove the following result.

\begin{theorem}\label{T:Lie groups boundedness properties}
Let $G$ be a semisimple Lie group.  
Then  $G$ is finitely normally generated, and
\begin{enumerate}[label=(\alph*)]
\item $G$ is bounded if and only if  $Z(G)$ is finite. 
\label{i:ZG infinite => unbounded}
\end{enumerate}
If $Z(G)$ is finite then the following hold. 
\begin{enumerate}[label=(\alph*)]
\setcounter{enumi}{1}
\item 
\label{i:G compact => not strongly bounded}
If $G/Z(G)$ has a non-trivial compact factor then $G$ is bounded but not strongly bounded. 

\item 
\label{i:G non compact => uniformly bounded}
If $G/Z(G)$ has no non-trivial compact factors then $G$ is uniformly bounded. 
\end{enumerate}
\end{theorem}

\begin{remark}
Dowerk and Thom proved that topologically simple compact groups are bounded
\cite[Proposition 2.2]{MR3907832}.
\end{remark}

Any connected Lie group $G$ acts on itself by conjugation and this gives rise
to the adjoint representation $\Ad \colon G \to \GL(\F g)$ whose kernel is $Z(G)$.
We will need the following standard fact.

\begin{lemma}\label{P:G simple => Ad irreducible}
Let $G$ be a connected Lie group with Lie algebra $\F g$.  
If $V \subseteq \F g$ is an $\OP{Ad}(G)$-invariant subspace then $V$ is an ideal in $\F g$.
\qed
\end{lemma}

%

\begin{lemma}\label{L:bg2n contains nbhd}
Let $G$ be a simple Lie group of dimension $n$.
If $g \in G$ is not in $Z(G)$ then $B_g(2n)$ contains an open neighbourhood $U$ of $1 \in G$ such that $U=U^{-1}$.
In particular, $g$ normally generates $G$.  
\end{lemma}

\begin{proof} 
Since $g \notin Z(G)$ and $G$ is connected, $\OP{Ad}(g) \in \OP{GL}(\F g)$ is not the identity transformation and therefore $\OP{Ad}(g)(Y)\neq Y$ for some $Y\in \F g$.  
Set $X=\OP{Ad}(g)(Y)-Y$.
Then $X \neq 0$ and the simplicity of $\F g$ and Lemma \ref{P:G simple => Ad irreducible} imply that $\OP{Ad}(G)(X)$ spans $\F g$.
Therefore there exist $g_1,\dots,g_n \in G$ such that \[ \OP{Ad}(g_1)(X),
\dots, \OP{Ad}(g_n)(X)
\]
form a basis of $\F g$.
Consider the smooth function $\Psi\colon \B R^n \to G$ given by
\[
\Psi \colon (t_1,\ldots,t_n)\mapsto [g,\exp(t_1Y)]^{g_1} \cdots [g,\exp(t_n Y)]^{g_n}.
\]
The differential of $\Psi$ at the origin satisfies
\[
d\Psi(\partial_i)=\OP{Ad}(g_i)(X).
\]
It follows that $\Psi$ is nondegenerate at $0 \in \B R^n$, hence its image
contains an open neighbourhood $U_g$ of the identity.  Since
$\Psi(t_1,\ldots,t_n)$ is a product of $2n$ conjugates of $g$, the image of
$\Psi$ is contained in $B_g(2n)$. 
Set $U=U_g \cap U_g^{-1}$.
\end{proof}

\begin{proof}[Proof of Theorem~\ref{thm:dense_Lie}]
 Any $g\not\in Z(G)$ normally generates $G$ by Lemma \ref{L:bg2n contains nbhd}.  Now $G$ can be equipped with a bi-invariant Riemannian metric \cite[Theorem~6.2]{MR3136522}, which gives rise to a conjugation-invariant norm (see Section~\ref{S:preliminaries}); this induces the usual topology on $G$.  Since $Z(G)$ is closed but not open, the result follows from Proposition~\ref{P:dense subgroups}\ref{I:dense subgroups:dense not strongly bounded}.
\end{proof}

The first step in proving Theorem \ref{T:Lie groups boundedness properties} is to investigate simple Lie groups with trivial center.
The compact case is straightforward.

\begin{proposition}\label{P:simple compact Z(G)=1}
Let $G$ be a simple compact Lie group with $Z(G)=1$.
Then $G$ is bounded and $\Delta_1(G)=\infty$.
\end{proposition}

\begin{proof}
The connectedness of $G$ and Lemma \ref{L:bg2n contains nbhd} imply that any $g \in G \setminus Z(G) \neq \emptyset$ normally generates, and together with the compactness of $G$ that $\|G\|_g < \infty$.
It follows from Corollary \ref{C:tfae bounded} that $G$ is bounded.


Let $d$ be any bi-invariant Riemannian metric on $G$ \cite[Theorem 16.2]{MR3136522}, and let $\nu$ be the associated conjugation-invariant norm; then $\nu$ induces the usual topology on $G$.
We have $\Delta_1(G)=\infty$ by Proposition \ref{P:dense subgroups}\ref{I:dense subgroups:Delta_1=infty}, taking the subgroup $N$ to be 1.
\end{proof}

The non-compact case is more involved.
Recall that the center of a simple Lie $G$ contains any proper normal subgroup of $G$.

\begin{lemma}\label{L:faithful repn of sl2r}
Let $G$ be a connected Lie group with $Z(G)=1$ and Lie algebra $\F g$.
Let $\psi \colon \F{sl}_2(\B R) \to \F g$ be an injective Lie algebra homomorphism.
Then there exists a smooth homomorphism $\varphi \colon \OP{SL(2,\B R)} \to G$ such that $L(\varphi)=\psi$.
\end{lemma}

\begin{proof}
Let $\tilde{S}$ be the universal cover of $\OP{SL}(2,\B R)$.  
Then $\tilde{S}$ is simple and any $N \triangleleft \tilde{S}$ is contained in $Z(\tilde{S})$ \cite[Prop. 11.1.4]{\HigNeeb}. 

By the Integrability Theorem of Lie algebra homomorphisms \cite[Theorem 9.5.9]{MR3025417} there exists $\tilde{\varphi} \colon \tilde{S} \to G$ such that $L(\tilde{\varphi})=\psi$.
Let $H$ denote its image and $i \colon H \to G$ the inclusion.
Then $H \cong \tilde{S}/D$ for some $D \leq Z(\tilde{S})$.
Since $Z(G)=1$, $\OP{Ad} \colon G \to \OP{GL}(\F g))
$ is injective.
Thus, $\OP{Ad} \circ i$ is a faithful finite-dimensional representation of $H$ and \cite[Example 16.1.8]{MR3025417} shows that $|Z(\tilde{S})/D| \leq 2$, hence $H\cong \OP{SL}(2,\B R)$ or $H \cong \OP{PSL}(2,\B R)$, thus $\tilde{\varphi}$ factors through $\varphi \colon \OP{SL}(2,\B R) \to G$.
\end{proof}

An element $X$ of a Lie algebra $\F g$ is called {\em nilpotent} if $\OP{ad}(X)\colon \F g \to \F g$ is a nilpotent linear map, i.e., the matrix representing $\OP{ad}(X)$ in some basis of $\F g$  is strictly lower (or upper) triangular.
Let $\Nil(\F g)$ denote the set of nilpotent elements in $\F g$.  Our next result strengthens Lemma~\ref{L:bg2n contains nbhd} because the open set $U$ we obtain does not depend on the choice of $g$.
\begin{proposition}\label{P:simple non-compact Z(G)=1}
Let $G$ be a non-compact simple Lie group of dimension $n$ with $Z(G)=1$ and Lie algebra $\F g$.
\begin{enumerate}[label=(\roman*)]
\item
\label{I:simple non-compact Z(G)=1:Bg nil}
$\exp_G(\Nil(\F g)) \subseteq B_g(2n)$ for every $1 \neq g \in G$.

\item
\label{I:simple non-compact Z(G)=1:Bg U}
There exists a neighbourhood $U=U^{-1}$ of the identity such that $U \subseteq B_g(4n^2)$ for any $g \neq 1$.

\item
\label{I:simple non-compact Z(G)=1:Delta}
$\Delta(G)<\infty$.
\end{enumerate}
\end{proposition}

\begin{proof}
\ref{I:simple non-compact Z(G)=1:Bg nil}
Fix $g \neq 1$ and let $X \in \Nil(\F g)$.
By the Jacobson-Morozov Theorem  \cite[Ch. VIII, \S 11.  2, Prop. 2]{MR2109105} $X$ is part of an $\F{sl}_2$-triple $(X,Y,H')$.
Let $e,f,h$ be the standard generators of $\F{sl}_2(\B R)$ \cite[Section 6.2]{MR3025417} and let $\psi \colon \F{sl}_2(\B R) \to \F g$ be the homomorphism defined by mapping the triple $(e,f,h)$ to the triple $(X,Y,H')$.
By Lemma \ref{L:faithful repn of sl2r} there exists a smooth homomorphism $\varphi \colon \SL(2,\B R) \to G$ such that $L(\varphi)=\psi$.
Conjugation by diagonal matrices shows that the closure of the orbit of $e=\left(\begin{smallmatrix}0 & 1 \\ 0 & 0  \end{smallmatrix}\right)$ under the adjoint action of $\OP{SL}(2,\B R)$ contains $0 \in \F{sl}_2(\B R)$ and therefore the conjugacy class of $\exp_{\OP{SL}(2,\B R)}(e)$ contains the identity matrix in its closure.
By the continuity and naturality of $\exp$, the conjugacy class of $\exp_G(X)$ contains $1 \in G$ in its closure (see \cite[Equation (1.82)]{MR1920389}), and therefore it intersects $B_g(2n)$ non-trivially by Lemma \ref{L:bg2n contains nbhd}.
Since $B_g(2n)$ is closed under conjugation, $\exp_G(X) \in B_g(2n)$.

\ref{I:simple non-compact Z(G)=1:Bg U}
By \cite[Theorem 5.1]{MR364552} $G=KNK$, where $G=KAN$ is the Iwasawa decomposition associated to $\F g = \F k + \F a +\F n$, see \cite[Section VI.4]{MR1920389} or \cite[Section 13.3]{MR3025417}.
By \cite[Theoren 6.31(f)]{\Knapp} $K$ is compact.
By \cite[Theorem 13.3.8 and Lemma 13.3.5]{\HigNeeb}, $N=\exp_G(\F n)$  and $\F n \subseteq \Nil(\F g)$.
Since $G$ is not compact, $N \neq 1$, hence $\Nil(\F g) \neq 0$.

Choose some $0 \neq X \in \Nil(\F g)$.
Since $\exp_G$ is a local diffeomorphism at $0$, by replacing $X$ with a scalar multiple, we may assume that $\exp_G(X) \neq 1$ and we may fix a neighbourhood of the identity $U \subseteq B_{\exp(X)}(2n)$ guaranteed by Lemma \ref{L:bg2n contains nbhd}.
For any $g \neq 1$ part \ref{I:simple non-compact Z(G)=1:Bg nil} shows that $\exp(X) \in B_g(2n)$ and it follows from Lemma \ref{L:BXn properties}\ref{L:BXn:assoc} that $U \subseteq B_g(2n \cdot 2n)$.

\ref{I:simple non-compact Z(G)=1:Delta}
We use the Iwasawa decomposition $G=KAN$ and $U$ from part \ref{I:simple non-compact Z(G)=1:Bg U}.
Since $K$ is compact and $G$ is connected, $K \subseteq U^{r}=U \cdot U \cdots U$ for some $r\in \B N$.
Consider an arbitrary $1 \neq g \in G$.
It follows from Lemma \ref{L:BXn properties} that $K \subseteq B_U(r) \subseteq B_g(4n^2r)$.
Part \ref{I:simple non-compact Z(G)=1:Bg nil} shows that $N=\exp_G(\F n) \subseteq \exp_G(\Nil(\F g)) \subseteq B_g(2n)$.
Combine this with \cite[Theorem 5.1]{MR364552} which asserts that $G=KNK$, to deduce that $\|G\|_g \leq 8n^2r+2n$, which is independent of $g$.
Since $g \in G$ was arbitrary, 
Lemma \ref{L:bg2n contains nbhd} combined with Proposition \ref{P:dense subgroups}\ref{I:dense subgroups:Delta_1=Delta} shows that $\Delta(G) = \Delta_1(G) < \infty$.
\end{proof}

If $G$ is semisimple then $Z(G)$ is a discrete subgroup and $G/Z(G)$ has trivial centre by \cite[Proposition 6.30]{\Knapp}.
Moreover, $G/Z(G)$ is the product of simple Lie groups with trivial centre.
This follows from elementary covering space theory (see \cite[Theorems 9.5.4]{MR3025417}) and the Integrability Theorems \cite[Theorems 9.4.8 and  9.5.9]{MR3025417} which imply that the universal cover $\tilde{G}$ of $G/Z(G)$ is the product of simply connected simple Lie groups, hence this is the case for $\tilde{G}/Z(\tilde{G})$.

Also, $Z(G)$ is a finitely generated abelian group.
This follows by combining \cite[Theorems 9.5.4 and 13.1.7]{MR3025417} and \cite[Proposition 6.30 and Theorem 6.31]{\Knapp} which show that $Z(G)$ is isomorphic to a subgroup of $\pi_1(G/Z(G)) \cong \pi_1(K)$, where $K$ is a maximal compact (Lie) subgroup of $G/Z(G)$.

Let $G$ be a Lie group and $A$ a $G$-module equipped with a metric (in this paper we will only be interested in the case of the trivial action of $G$).
One can study the (bounded) continuous cohomology groups $H^*_c(G,A)$ and $H^*_{cb}(G,A)$ defined by means of continuous cochains $f \colon G^p \to A$.
The open sets in $G$ and $A$ define the Borel $\sigma$-algebras on $G$ and $A$ and one can consider the (bounded) Borel cohomology groups $H^*_B(G,A)$ and $H^*_{Bb}(G,A)$ defined by means of the cochains $f \colon G^p \to A$ that are (bounded) Borel maps.
There are obvious inclusion of cochain complexes which give rise to comparison maps between these cohomology groups and which fit into the commutative diagram
\[
\xymatrix{
H^*_{cb}(G,A) \ar[r]^{\iota_*} \ar[d]_{j_*} &
H^*_{Bb}(G,A) \ar[d]^{j_*} \\
H^*_c(G,A) \ar[r]^{\iota_*} &
H^*_B(G,A).
}
\]
A nice survey 
can be found in \cite[\S 2-4]{MR0494071} and in Moore's  paper \cite{MR0171880}.

\begin{proposition}\label{P:infinite cyclic centre}
Let $H$ be a connected semisimple Lie group.
Assume that $Z=Z(H)$ is an infinite cyclic group.
Then there exists an unbounded quasimorphism $q \colon H \to \B R$.
\end{proposition}

\begin{proof}
Set $G=H/Z$ and let $\pi \colon H \to G$ be the quotient map.
Then $Z(G)$ is trivial and $G$ is a product of centre-free simple Lie groups.
In particular the abelianisation of $G$ is trivial.

Consider the short exact sequence of trivial $G$-modules
\[
0 \to \B Z \xto{i} \B R \to T \to 0.
\]
There results a long exact sequence in Borel cohomology \cite[p. 43]{MR0171880}
\[
\dots \to H^1_B(G,T) \to H^2_B(G,\B Z) \xto{i_*} H^2_B(G,\B R) \to H^2_B(G,T) \to \dots.
\]
Also, since $G$ acts on $\B R$ trivially it follows from \cite[p.\ 45]{MR0171880} that  $H^1_B(G,T)$ is isomorphic to the group of continuous homomorphisms $G \to T$, and since $G_{ab}$ is trivial, $H^1_B(G,T)=0$.
It follows that $H^2_B(G,\B Z) \xto{i_*} H^2_B(G,\B R)$ is injective.

The comparison maps between (bounded) continuous and (bounded) Borel cohomology and the naturality with respect to group homomorphisms give rise to the following commutative diagram:
\[
\xymatrix{
H^2_{c}(H,\B R) \ar[d]_{\iota_*}^\cong &
H^2_{c}(G,\B R) \ar[d]_{\iota_*}^{\cong} \ar[l]_{\pi^*} &
H^2_{cb}(G,\B R) \ar[d]_{\iota_*}^{\cong} \ar[r]^{\pi^*}_\cong  \ar[l]_{j_*} &
H^2_{cb}(H,\B R) \ar[d]_{\iota_*}^\cong
\\
H^2_B((H,\B R) &
H^2_B(G,\B R) \ar[l]_{\pi^*} &
H^2_{Bb}(G,\B R) \ar[l]_{j_*} \ar[r]^{\pi^*} &
H^2_{Bb}(H,\B R) \\
H^2_B(H,\B Z) \ar[u]^{i^*} &
H^2_B(G,\B Z) \ar[l]_{\pi^*} \ar@{>->}[u]^{i_*} &
H^2_{Bb}(G,\B Z) \ar[u]^{i_*} \ar[r]^{\pi^*} \ar[l]_{j_*} &
H^2_{Bb}(H,\B Z) \ar[u]^{i_*}
}
\]
The first two vertical arrows in the first row are isomorphisms by \cite[Theorem A]{MR3063901}.
The last two are isomorphisms by \cite[Item (2.i) in Section 2.3, p.\ 529] {MR2680425}.
This result uses the regularization operator $R^*$ in \cite[Section 4]{MR543215}; it is immediate from its definition and from the definition of the chain homotopies in {\em loc.\ cit.}\ that $R^*$ restricts to cochain equivalences $R^* \colon C_{Bb}^*(G,\B R) \to C_{cb}^*(G,\B R)$ and $R^* \colon C_{Bb}^*(H,\B R) \to C_{cb}^*(H,\B R)$ of the bounded cochain complexes.
See also the remarks in \cite[p. 553]{MR2854105}.
The last horizontal arrow in the first row is an isomorphism by \cite[Corollary 7.50.10]{MR1840942} since $\ker(\pi)=Z$ is amenable.
The second vertical arrow in the second row is injective as we have seen above.

Since $H$ is connected, the central extension $Z \to H \to G$ is not trivial (i.e., not split).
Also, by \cite[p. 45]{MR0171880} or \cite{MR0089998}, $H^2_B(G,Z)$ is isomorphic to $\OP{Ext}(G,Z)$, the group of equivalence classes of extensions of topological groups.
Hence $H$ gives rise to a non-trivial class $[\epsilon^H] \in H^2_B(G,\B Z)$.
By \cite[Theorem 1.1]{MR2854105} the arrow $j_*$ in the last row is surjective, hence there exists $[\epsilon^H_b] \in H^2_{Bb}(G,\B Z)$, a preimage of $[\epsilon^H]$.
Let $[f] \in H^2_B(G,\B R)$ and $[f_b] \in H^2_{Bb}(G,\B R)$ be the images of $[\epsilon^H]$ and $[\epsilon^H_b]$ under $i^*$.
Since the second vertical arrow $i_*$ in the second row of the diagram is injective, $[f] \neq 0$, and since $j_*([f_b])=[f]$, also $[f_b] \neq 0$.
Since the vertical maps $\iota_*$ are isomorphisms, there exist $[f_c] \in H^2_c(G,\B R)$ and $[f_{cb}] \in H^2_{cb}(G,\B R)$ such that $\iota_*[f_c]=[f]$ and $\iota_*[f_{cb}]=[f_b]$.

It is a standard fact that $\pi^*[\epsilon^H]$ is the trivial element in $H^2_B(H,\B Z)$ (because $\epsilon^H$ is the coboundary of the Borel section $\sigma \colon G \to H$ used to define $\epsilon^H$).
It follows that $\pi^*[f_c]=0$ since the $\iota_*$ are isomorphisms and
\[
\iota_*\pi^*[f_c]=\pi^*[f] = \pi^* i_*[\epsilon^H]=i_*\pi^*[\epsilon^H]=0.
\]
Since $\iota_*$ and $\pi^*$ at the top right-hand corner of the diagram are isomorphisms, 
\[
\pi^*[f_{cb}] \neq 0 \qquad \text{  (in $H^2_{cb}(H,\B R)$)}.
\]
The first row of the diagram above is part of the commutative diagram
\[
\xymatrix{
H^2_{cb}(G,\B R) \ar[r]^{j_*} \ar[d]^{\pi^*} &
H^2_c(G,\B R) \ar[d]^{\pi^*} 
\\
H^2_{cb}(H,\B R) \ar[r]^{j_*} &
H^2_c(H,\B R).
}
\]
Hence $j_*(\pi^*[f_{cb}])=\pi^*(j_*[f_{cb}])=\pi^*[f_c]=0$.
Therefore there exists a continuous map $q \colon H \to \B R$ such that $\pi^*f_{cb}=\partial q$.
It must be unbounded, or else $[\pi^*f_{cb}]=0$ which is a contradiction.
It is also a quasimorphism: for $f_{cb}$ is a bounded $2$-cocycle on $G$, so for any $h_1, h_2 \in H$ we get
\begin{multline*}
|q(h_1 h_2)-q(h_1)-q(h_2)| = |\partial q(h_1,h_2)| = |(\pi^*f_{cb})(h_1,h_2)| = \\
|f_{cb}(\pi(h_1),\pi(h_2))| < M,
\end{multline*}
where $M$ is a bound for $f_{cb}$.
\end{proof}

\begin{proof}[Proof of Theorem \ref{T:Lie groups boundedness properties}]
Set $H=G/Z(G)$.
Then $Z(H)=1$ and $H=H_1 \times \dots \times H_k$ is a product of simple Lie groups with trivial center.
Apply Lemma \ref{L:bg2n contains nbhd} to each $H_i$ to deduce that $H$ is finitely normally generated.
It follows that $G$ is finitely normally generated because $Z(G)$ is finitely generated.

Suppose that $|Z(G)|<\infty$.
Each factor $H_i$ is bounded by Propositions \ref{P:simple compact Z(G)=1} and \ref{P:simple non-compact Z(G)=1}\ref{I:simple non-compact Z(G)=1:Delta}.
Therefore $H$ is bounded by Lemma \ref{L:product of simple uniformly bounded}\ref{I:fcg bounded for product} and it follows from Lemma \ref{L:finite_extn new}\ref{I:finite_extn new:1} that $G$ is bounded.
This proves the ``if'' statement in \ref{i:ZG infinite => unbounded}.

Suppose that $H$ a compact factor, say $H_1$.
Then $\Delta_1(H_1) = \infty$ by Proposition \ref{P:simple compact Z(G)=1} and Lemma \ref{L:finite gen quotient not strongly bounded}\ref{I:finite gen quotient not strongly bounded:2} shows that $G$ is not strongly bounded.
This proves \ref{i:G compact => not strongly bounded}.

Suppose that $H_1,\dots,H_k$ are not compact.
Then $\Delta(H_i)<\infty$ by Propostion \ref{P:simple non-compact Z(G)=1}\ref{I:simple non-compact Z(G)=1:Delta} and Lemma \ref{L:product of simple uniformly bounded}\ref{I:u bounded for product of simple} shows that $\Delta(H)<\infty$ (because each $H_i$ is simple as an abstract group by Lemma \ref{L:bg2n contains nbhd}).
Then $\Delta(G)<\infty$ by Lemma \ref{L:finite_extn new}\ref{I:finite_extn new:2}.
This proves \ref{i:G non compact => uniformly bounded}.

Assume that $Z(G)$ is infinite.
Since $Z(G)$ is finitely generated, it contains a factor isomorphic to $\B Z$ with complement $Z(G)'$.
Set $K=G/Z(G)'$.
Then $Z(K) \cong \B Z$ (because $Z \leq G$ is discrete and closed).
By Lemma \ref{L:finite gen quotient not strongly bounded} it suffices to prove that $K$ is unbounded.
This follows from Proposition~\ref{P:infinite cyclic centre}  and from \cite[Lemma 3.6]{MR2819193} which implies that all quasimorphisms on a bounded finitely normally generated group must be bounded.
\end{proof}

%
%
\section{Linear algebraic groups}\label{S:lin-alg-groups}


Throughout this section $G$ denotes a linear algebraic group \emph{over an algebraically closed field $K$}.  
The Zariski closure of a subset $X$ of an algebraic variety over $K$ is denoted $\overline{X}$.  
If $A$ is a constructible dense subset of an irreducible variety $X$ then $A$ contains an open subset of $X$, \cite[AG.1.3]{MR1102012}. 
Indeed, $A = \bigcup_{i=1}^n F_i \cap U_i$ where the $F_i$ are closed and the $U_i$ are open, so $X=\overline{A} \subseteq \bigcup_{i=1}^n F_i$, and since $X$ is irreducible $F_i=X$ for some $i$.  
If $G$ is any algebraic group and $A,B \subseteq G$ then $\overline{A} B \subseteq \overline{AB}$ because $\overline{A}b = \overline{Ab} = \overline{AB}$ for any $b \in B$.

The next lemma is a slight improvement on \cite[7.5 Proposition]{MR0396773} and its proof, which we follow closely. Our addition is the upper bound for $k$. 

\begin{lemma}
\label{L:commutator_gen2}
 Let $G$ be a linear algebraic group over an algebraically closed field $K$.  Let
$\{f_i \colon V_i \to G\}_{i \in I}$ be a family of morphisms from irreducible
varieties $V_i$ such that $1 \in W_i:= f_i(V_i)$ for every $i \in I$.  Let $G'$
be the closed subgroup generated by $\bigcup_{i \in I} W_i$.  Then $G'$ is a
connected closed subgroup of $G^0$ and there are sequences $i_1,\dots,i_k \in
I$ and $e_1,\dots,e_k \in \{\pm 1\}$ for some $k \leq 2\OP{dim} G$, such that
$G'=W_{1_1}^{e_1}\cdots W_{i_k}^{e_k}$.
\end{lemma}

\begin{proof}
We may assume that $G' \neq \{1\}$.
Let us construct by induction a sequence $i_1,i_2,\dots$ in $I$ and a sequence $e_1,e_2,\dots$ in $\{\pm 1\}$ as follows.
Choose any element $i_1 \in I$ such that $W_i \neq \{1\}$ and choose $e_1=1$.
Assume that $i_1,\dots,i_m$ and $e_1,\dots,e_m$ have been chosen.
Choose $i_{m+1}$ and $e_{m+1}$ as follows.
Set
\[
D_m=W_{i_1}^{e_1} \cdots W_{i_m}^{e_m}.
\]
Then $D_m \subseteq G'$ is the image of a morphism of varieties $V_{i_1} \times \dots \times V_{i_m}  \to G \times \cdots \times G \xrightarrow{\text{mult}} G$ and it is therefore an irreducible constructible subset of $G$ \cite[Theorem 4.4 and Proposition A in \S1.3]{MR0396773}.  
Moreover, $1 \in D_m$ since $1 \in W_i$ for all $i$, so $D_m \subseteq G^0$.  
If $\overline{D_m}=G'$ then set $i_{m+1}=i_m$ and $e_{m+1}=e_m$.
So assume that $\overline{D_m} \subsetneq G'$.
There must exist $j_1,\dots,j_r \in I$ such that $W_{j_1}^{\pm 1} \cdots W_{j_r}^{\pm 1} \nsubseteq \overline{D_m}$ (since otherwise $G' = \overline{D_m}$).  
Since $1 \in D_m$ and $1 \in W_{j_1},\ldots,W_{j_r}$, it follows that $\overline{D_m} W_{j_1}^{\pm 1} \cdots W_{j_r}^{\pm 1} \supsetneq \overline{D_m}$ and therefore $\overline{D_m} W_{j_q}^{h_q}  \supsetneq \overline{D_m}$ for some $1 \leq q \leq r$ and some $h_q=\pm 1$.
Choose $i_{m+1}=j_q$ and $e_{m+1}=h_q$.

We have seen above that the sets $D_m$ 
are irreducible and constructible for all $m \geq 1$.  
Also $1 \in D_m$.  
In addition, if $\overline{D_m} \neq G'$ then by construction of $i_{m+1}$, 
\[
\overline{D_{m+1}}=\overline{D_m W_{i_{m+1}}^{e_{m+1}}} \supseteq \overline{D_m} W_{i_{m+1}}^{e_{m+1}} \supsetneq \overline{D_m}.
\]
Since the sets $\overline{D_m}$ are closed subsets of the affine variety $G$, they are affine varieties.
By \cite[Proposition 3.2]{MR0396773}, $\OP{dim}(\overline{D_m}) < \OP{dim}(\overline{D_{m+1}})$. 
Since $\OP{dim}(\overline{D_m}) \leq \OP{dim}(G')$ 
for all $m$ we deduce from the construction of $\{ D_n\}_{n \geq 1}$ that $\overline{D_n}=G'$ for some $n$.
In particular $G'$ is connected, thus it is a closed subgroup of $G^0$. 
Also $n \leq \OP{dim}(G^0)$, since $D_1 \neq \{1\}$.

Finally, $D_n$ is constructible and dense in $G'$, so it contains a nonempty open subset of $G'$.  Hence $D_n D_n=G'$ by \cite[7.4 Lemma]{MR0396773}.
Thus, $G'=(W_{i_1}^{e_1} \cdots W_{i_n}^{e_n}) (W_{i_1}^{e_1} \cdots W_{i_n}^{e_n})$ as needed.
\end{proof}

Recall from \cite[Theorem II.6.8]{MR1102012} that if $N$ is a closed normal subgroup of a linear algebraic group $G$ then $G/N$ is also a linear algebraic group.  
The commutator subgroup $[G,G]$ is closed by \cite[Section I.2.3]{MR1102012}, as is $[G,G^0]$.

\begin{proposition}\label{P:alggp}
Let $G$ be a linear algebraic group over an algebraically closed field $K$.
The following are equivalent.
\begin{itemize}
\item[(a)] $G$ is finitely normally generated;
\item[(b)]  $G/[G,G]$ is finite;
\item[(c)] $G/[G,G^0]$ is finite.
\end{itemize}
\end{proposition}

\begin{proof} 
$(a)\implies (b)$.
If $G$ is finitely normally generated then $A=G/[G,G]$ is a finitely normally generated abelian group, hence it is finitely generated. 
Suppose $A$ is infinite and pick $n>|A/A^0|$.  Define
a homomorphism $f\colon A\to A$ by $f(a)=a^n$.  Then the image $f(A)$ is a
closed subgroup of $A$ \cite[7.4 Proposition B(b)]{MR0396773} and 
$[A:f(A)]\geq n$ because $A$ has $\B Z$ as a direct factor.  
But $A^0\subseteq f(A)$ by \cite[7.3 Proposition (b)]{MR0396773}, so $|A:f(A)| \leq |A:A^0|<n$  which is absurd.
Therefore $A$ must be finite.

$(b)\implies (c)$.
Assume that $G/[G,G]$ is finite.
Set $H=G/[G,G^0]$ and let $\pi \colon G \to H$ be the projection.  By
construction of $H$ and \cite[7.4 Proposition B(c)]{MR0396773}, 
$H^0=\pi(G^0) \subseteq Z(H)$, and therefore $[H,H]$ is finite by 
\cite[17.1 Lemma A]{MR0396773}.  Since $H/[H,H]$ is a quotient of $G/[G,G]$, it
is finite and therefore $H$ is finite.

$(c)\implies (a)$. 
For any $g \in G$ consider the morphism of varieties $f_g \colon G^0 \xrightarrow{x \mapsto [g,x]} G$.
Then $1\in \OP{Im}(f_g)$ and $[G,G^0]$ is the closed subgroup generated by $\bigcup_{g \in G} \OP{Im}(f_g)$.
Lemma \ref{L:commutator_gen2} implies that there is a finite $T\subseteq G$ such that $[G,G^0]$ is generated by $\bigcup_{g \in T}\OP{Im}(f_g)$, hence it is generated by $\OP{conj}_{G^0}(T)$.
If $G/[G,G^0]$ is finite  then $G$ is normally generated by $T$ and any set
of representatives in $G$ for the cosets of $[G,G^0]$.
\end{proof}

\begin{theorem}
\label{T:alggp_unif_bd}
 Let $G$ be a finitely normally generated linear algebraic group over an algebraically closed field $K$.
Then $G$ is uniformly bounded and $\Delta(G)\leq 4\,{\rm dim}(G)+ \Delta(G/G^0)$.
\end{theorem}

\begin{proof}
%
%
 Let $T \subseteq G$ be a finite normally generating set.
Then $G$ is generated by $\OP{conj}_G(T)$.  Any $x \in \OP{conj}_G(T)$ yields a morphism of varieties $f_x \colon G^0 \xrightarrow{y \mapsto [x,y]} G$; set $W_x:= f_x(G^0)$.  Let $S=\bigcup_{x \in \OP{conj}_G(T)} W_x$, and let $N$ be the subgroup of $G$ generated by $S$.  Clearly $S$ is invariant under conjugation by $G$, so $N \unlhd G$.  Every element of $W_x$ is a product of two conjugates of $x$, so $W_x \subseteq B_T(2)$.  It follows from Lemma~\ref{L:commutator_gen2} that $N$ is closed and $N\subseteq B_T(4 \OP{dim}(G))$.

 Let $\pi \colon G \to G/N$ be the canonical projection.  By construction $\pi(G^0)$ commutes with every element of $\pi(\OP{conj}_G(T))$.  Since the latter generates $G/N$ it follows that $\pi(G^0) \subseteq Z(G/N)$ and therefore $N \supseteq [G,G^0]$.  
By Lemma \ref{P:alggp}, $[G,G^0]$ has finite index in $G$, so $N$ has finite index in $G$.
It follows that $G^0 \leq N$ \cite[7.3 Proposition (b)]{MR0396773}.

 The image of $T$ in $G/G^0$ normally generates $G/G^0$, so $B_T(\Delta(G/G^0))$ contains an element from every coset of $G^0$ in $G$.
Since $G^0 \leq N$ we deduce that $G=B_T(4 \OP{dim}(G)+\Delta(G/G^0))$.
But $T$ was arbitrary, so the result follows.
\end{proof}

\begin{remark}
Theorem \ref{T:alggp_unif_bd} fails if the ambient field $K$ is not algebraically closed.  
For example, for $n\geq 3$, the real algebraic Lie group $\OP{SO}(n,\B R)$ is a simple compact Lie group, which by Theorem \ref{T:Lie groups boundedness properties} is not strongly bounded.
\end{remark}


%
%
\section{Non-uniformly bounded groups and simple quotients}\label{S:not-UB}

The purpose of this section is to prove the following theorem.

\begin{theorem}
\label{T:simple_quotient}
Let $G$ be a finitely normally generated group and suppose $G$ has infinitely many maximal normal subgroups.  
Then $G$ is not uniformly bounded.
\end{theorem}

\begin{definition}
Let $\C N$ be a collection of normal subgroups of a group $G$.
Set $G/\C N = \prod_{N \in \C N} G/N$.
The collection is called {\em splitting} if every $N \in \C N$ is a proper subgroup and the natural homomorphism $G \to G/\C N$ is surjective.
\end{definition}

Clearly, a sub-collection of a splitting collection is splitting and $\C N = \{N\}$ is splitting for any proper $N \triangleleft G$.

\begin{lemma} \label{L:splitting new}
Let $G$ be a finitely normally generated group which admits a splitting collection $\C N$ of size $k$. 
Then $\Delta(G) \geq k$.
\end{lemma}

\begin{proof}
Since $\Delta(G/N) \geq 1$ for all $N \in \C N$, Lemmas \ref{L:product of simple uniformly bounded}\ref{I:product delta geq sum delta} and \ref{L:finite gen quotient not strongly bounded} show that $k \leq \Delta(G/N_1 \times \dots \times G/N_k) \leq \Delta(G)$.
\end{proof}

The next lemma is straightforward; it follows from \cite[Exercise 404(i), page 167]{MR0498810}, for example.

\begin{lemma}\label{L:finitely many maximal normal subgroups}
Suppose $G=G_1 \times \dots \times G_n$ is a product of simple groups.
Then $G$ has only finitely many normal subgroups.
\end{lemma}

\begin{proof}[Proof of Theorem \ref{T:simple_quotient}]
Let $\C M(H)$ denote the collection of the maximal normal subgroups of a group $H$.
Let $\B M(G)$ denote the set of all the splitting collections of $G$ consisting of maximal normal subgroups.
Clearly $\B M(G)$ is not empty.
Suppose $\C N \in \B M(G)$ has size $k\in \B N$.
Let $K$ be the kernel of $G \to G/\C N$.
If $K$ is contained in every $M \in \C M(G)$ then the assignment $M \mapsto M/K$ gives rise to a bijection $\C M(G) \to \C M(G/K) \approx \C M(G/\C N)$ which contradicts Lemma \ref{L:finitely many maximal normal subgroups}.
So there exists $M \in \C M(G)$ such that $K \nsubseteq M$.
In particular $M \notin \C  N$.
Set $\C N'=\C N \cup \{ M\}$.
Then $G \to G/ \C N' = G/\C N \times G/M$ is surjective because $G \to G/\C N$ is surjective and $K \to G/M$ is surjective since it is not trivial and $G/M$ is simple.
Therefore $\C N' \in \B M(G)$.
Thus, $\B M(G)$ contains elements of arbitrary size $k \in \B N$ and Lemma \ref{L:splitting new} completes the proof.
\end{proof}


Here is an application to linear groups.

\begin{theorem}\label{T:strong-approximation}
Let $G\subset \OP{GL}(n,\B C)$ be a connected, simply connected, $\B
Q$-simple and absolutely simple linear algebraic group defined over $\B Q$. Let 
$\Gamma\subset G(\B Q)$ be a finitely generated Zariski dense subgroup.  Then
$\Gamma$ is not uniformly bounded.
\end{theorem}

\begin{proof}
It follows from the Strong Approximation Theorem \cite[Corollary 16.4.3]{MR1978431}
that $G(\B F_p)$ is a quotient of $\Gamma$ for almost all primes $p$. 
Moreover, our hypotheses imply that $G(\B F_p)$ is quasisimple \cite[Proposition 6.1]{0803.4165v5}, so
its quotient by its centre is a finite simple group.  
This yields infinitely many pairwise non-isomorphic simple quotients of $G$
\cite[Proposition 16.4.2]{MR1978431}.
The kernels of the quotient maps must be pairwise distinct maximal normal subgroups of $G$.
Theorem  \ref{T:simple_quotient} applies.
\end{proof}

\begin{remark}
 Theorem \ref{T:strong-approximation} fails if we drop the requirement that $\Gamma$ is finitely generated.  For instance, if $G= \OP{SL}(n)$ and $\Gamma= \OP{SL}(n, \B Q)= G(\B Q)$ then $\Gamma$ is uniformly bounded by Theorem~\ref{T:slnr main theorem:2}.
\end{remark}

\section{Boundedness properties of $\OP{SL}(n,\F R)$}\label{S:dslnz}

\noindent

\noindent
{\bf Assumptions and notation.}
In this section, unless otherwise stated, $\F R$ denotes a principal ideal domain (p.i.d).
The ideal generated by $a \in \F R$ is denoted $(a)_{\F R}$.
The {\em greatest common divisor} of $X \subseteq \F R$, denoted by $\gcd(X)$, is a generator of the ideal  $\sum_{a \in X} (a)_{\F R}$.

The group of $n \times n$ matrices with determinant $1$ over a commutative ring $R$ with 1 is denoted $\OP{SL}(n,R)$.
If $A,B\in \OP{SL}(n,R)$ then we write $A\sim B$ if $A$ is conjugate to $B$.
Let $e_{i,j}$ be the $n \times n$ matrix over $R$ whose only non-zero entry is $1$ in the $i$th row and the $j$th column where $i \neq j$.
For any $r \in R$ and $i \neq j$, the \emph{elementary matrix} $E_{i,j}(r)$ is $I+re_{i,j}$.
The set of all elementary matrices is denoted $\OP{EL}(n,R)$; it is clearly contained in $\OP{SL}(n,R)$.

We now state the main three results of this section.

\begin{theorem}\label{T:slnr main theorem}
Let $\F R$ be a p.i.d with infinitely many maximal ideals.
Let $n \geq 3$. 
Assume that $\OP{SL}(n,\F R)$ is normally generated by $\OP{EL}(n,\F R)$ and that $\| \OP{SL}(n,\F R)\|_{\OP{EL}(n,\F R)} \leq C_n$.
Then for any $k \geq 1$,
\[
k \leq \Delta_k(\OP{SL}(n,\F R)) \leq (4n+ 4)  C_n k.
\]
In particular $\OP{SL}(n,\F R)$ is strongly bounded but not uniformly bounded.
\end{theorem}

The assumptions in Theorem~\ref{T:slnr main theorem} that $\OP{SL}(n,\F R)$ is normally generated by $\OP{EL}(n,\F R)$  and that $\| \OP{SL}(n,\F R)\|_{\OP{EL}(n,\F R)}<\infty$ are not automatically satisfied for general p.i.d.'s (see Remark \ref{R:E does not generate SLn} below).

Application of Theorem~\ref{T:slnr main theorem} yields the following result.

\begin{corollary}\label{C:Carter Keller implications}
Let $\C O$ be the ring of integers in a number field whose class number is one and let $n \geq 3$.
Then $\OP{SL}(n,\C O)$ is normally generated by any elementary matrix $E_{i,j}(1)$, it is strongly bounded but not uniformly bounded.
In fact,
\[
k \leq \Delta_k(\OP{SL}(n,\C O)) \leq (4n+51) (4n+ 4) k.
\]
\end{corollary}

The situation is quite different when $\F R$ has only finitely many maximal ideals.

\begin{theorem}\label{T:slnr main theorem:2}
Let $\F R$ be a p.i.d with only $d<\infty$ maximal ideals.
Let $n \geq 3$.
Then $\OP{SL}(n,\F R)$ is normally generated by $\OP{EL}(n,\F R)$ and for any $k \geq 1$,
\[
\Delta_k(\OP{SL}(n,\F R)) \leq 12(n-1) \cdot \min\{ d, k(n+1)\}.
\]
In particular $\Delta(\OP{SL}(n,\F R)) \leq 12d(n-1)$, thus $\OP{SL}(n,\F R)$ is uniformly bounded. 
\end{theorem}


In the remainder of this section we will prove these results and provide some examples.

\begin{lemma}\label{L:orbits of sln}
Let $n \geq 2$.
Let ${\bf a}=(a_1,\ldots,a_n) \in \F R^n$ be a (row) vector.
Set $t=\gcd(a_1,\dots,a_n)$.
Then there exists $A \in \SL(n,\F R)$ such that ${\bf a} \cdot A = (t,0,\dots,0)$.
\end{lemma}

\begin{proof} 
Clearly, ${\bf a}=t{\bf b}$ where $\gcd({\bf b})=1$.
By \cite[Corollary II.1]{MR0340283} there exists $B\in \OP{SL}(n,\F R)$ whose first row is ${\bf b}$. 
Thus, $(t,0,0,\dots,0) \cdot B = t{\bf b}={\bf a}$.
Set $A=B^{-1}$. 
\end{proof}

Recall that $e_{i,j} e_{k,\ell}=0$ if $j \neq k$ and  $e_{i,j} e_{k,\ell}=e_{i,\ell}$ if $j=k$.
Suppose that $1 \leq i \neq \ell \leq n$.
For any $1 \leq j,k \leq n$ such that $i \neq j$ and $k \neq \ell$ we obtain the Steinberg relations
\[
[E_{i,j}(x),E_{k,\ell}(y)] = \left\{
\begin{array}{ll}
I & \text{if $j \neq k$} \\
E_{i,\ell}(xy) & \text{if $j=k$.}
\end{array}\right.
\]
Given $i \neq j$ set
\[
\sigma_{i,j} = e_{i,j}-e_{j,i}+\sum_{k \neq i,j} e_{k,k}.
\]
Notice that $\sigma_{i,j} \in \OP{SL}(n,\F R)$ and that $\sigma_{i,j}^{-1}=\sigma_{j,i}$.

\begin{lemma}\label{L:conjugate elementary matrices}
Assume that $n \geq 3$.
\begin{enumerate}
\item 
For any fixed $x \in \F R$, all the elementary matrices $E_{i,j}(x)$ are conjugate in $\OP{SL}(n,\F R)$.
\label{L:conjugate elementary matrices:1}

\item
Set $A=E_{1,n}(1)$.
Then $B_A(2) \supseteq \OP{EL}(n,\F R)$: that is, every elementary matrix is the product of at most $2$ conjugates of $A^{\pm 1}$. 
\label{L:conjugate elementary matrices:2}
\end{enumerate}
\end{lemma}

\begin{proof}
Consider $i \neq j$.
To prove (\ref{L:conjugate elementary matrices:1}) it suffices to show that $E_{i,j}(x) \sim E_{1,n}(x)$.  
Choose $k \neq i,j$ (this is possible since $n \geq 3$).
An easy calculation shows that $\sigma_{k,j}e_{i,j} \sigma_{k,j}^{-1}=e_{i,k}$ and therefore $E_{i,k}(x)=\sigma_{k,j}E_{i,j}(x) \sigma_{k,j}^{-1}$.
The rest is straightforward.
%
%
%
For part (\ref{L:conjugate elementary matrices:2}) use Steinberg's relation $E_{1,2}(x)=[E_{1,n}(1), E_{n,2}(x)]\in B_A(2)$ and part (\ref{L:conjugate elementary matrices:1}).
\end{proof}

\begin{lemma}\label{L:conjugate to elementary}
Fix $n \geq 2$, $2 \leq k \leq n$ and $a_1,\dots,a_{k-1} \in \F R$. 
Consider the upper triangular matrix 
\[
A= I + \sum_{i=1}^{k-1} a_i e_{i,k} =
\begin{pmatrix}
1 & 0 & \cdots & 0 & a_1    & 0 &\cdots & 0\\
  & 1 &        &   & a_2    &   &       &  \\
  &   & \ddots &   & \vdots &   &       &  \\
  &   &        & 1 & a_{k-1}&   &       &  \\
  &   &        &   & 1      &   &       &  \\
  &   &        &   &        & 1 &       &  \\
  &   &        &   &        &   &\ddots &  \\
  &   &        &   &        &   &       & 1
\end{pmatrix}.
\]
Then $A$ is conjugate to $E_{1,n}(t)$ where $t=\gcd(a_1,\dots, a_{k-1})$.
\end{lemma}

\begin{proof}
By the analogue of Lemma~\ref{L:orbits of sln} for column vectors, there exists a matrix $B \in \OP{SL}(n,\F R)$ of the form 
$\left(\begin{smallmatrix} D & 0 \\ 0 & I \end{smallmatrix} \right)$ where $D \in \OP{SL}(k-1,\F R)$
such that
$BAB^{-1} = E_{1,k}(t)$.
If $k \neq n$ then conjugation by $\sigma_{k,n}$ gives $E_{1,n}(t)$.
\end{proof}

The next lemma will be a key tool in our analysis.

\begin{lemma}[The double commutator lemma]\label{L:magic double commutators}
Let $R$ be a commutative ring with 1.
Let $A \in \SL(n,R)$ and set $B=A^{-1}$.
Write $A=(a_{i,j})$ and $B=(b_{i,j})$.
Fix indices $1 \leq i \neq \ell \leq n$ such that $a_{\ell,i}=0$.
Then for any $1 \leq j,k \leq n$ such that $j \neq i$ and $k \neq \ell$, 
and for any $x \in R$,
\[
[[ A,E_{i,j}(1)],E_{k,\ell}(x)] =
\left\{
\begin{array}{ll}
I+x b_{j,k} A e_{i,\ell} & \text{if } j \neq k, \\
I - x e_{i,\ell} + x(b_{j,j}-b_{j,i}) A e_{i,\ell} & \text{if } j = k. \\
\end{array}\right.
\]
\end{lemma}

\begin{proof}
For any $1 \leq s \neq t \leq n$ and any $D \in \SL(n,R)$ observe that 
$(De_{s,t}D^{-1})^2=De_{s,t}^2 D^{-1} = 0$.
Therefore
\[
(I+De_{s,t}D^{-1})^{-1} = I-De_{s,t}D^{-1}.
\]
Also, if we write $D=(d_{i,j})$, then for any $1 \leq p,q,r,s \leq n$,
\[
e_{p,q} D e_{r,s} = d_{q,r} e_{p,s}.
\]
By assumption $a_{\ell,i}=0$, hence for any $1 \leq j,t \leq n$,
\begin{eqnarray}
\nonumber
&& \qquad \mbox{\ \ \ } (I+Ae_{i,j}A^{-1}) e_{t,\ell} (I-Ae_{i,j}A^{-1}) \\
\nonumber
&& \qquad =
e_{t,\ell} - e_{t,\ell}Ae_{i,j} A^{-1} + 
Ae_{i,j}A^{-1} e_{t,\ell} - 
Ae_{i,j}A^{-1}e_{t,\ell}A e_{i,j} A^{-1}  \\
\nonumber
&& \qquad =
e_{t,\ell} - a_{\ell,i} e_{t,j} A^{-1} + 
b_{j,t}A e_{i,\ell} - a_{\ell,i} A e_{i,j} A^{-1} e_{t,j} A^{-1} \\
&& \qquad = e_{t,\ell} + b_{j,t}A e_{i,\ell}.
\end{eqnarray}
Therefore, if $t \neq \ell$ and $i \neq j$ then 
\begin{eqnarray}
\label{Eq:magic 2}
\nonumber
&& \qquad \mbox{\ \ \ \ } (I+Ae_{i,j}A^{-1}) E_{t,\ell}(x) (I-Ae_{i,j}A^{-1}) \\
\nonumber
&& \qquad = I + x(I+Ae_{i,j}A^{-1}) e_{t,\ell} (I-Ae_{i,j}A^{-1}) \\
&& \qquad = I + x e_{t,\ell} + x b_{j,t} A e_{i,\ell}.
\end{eqnarray}
It follows that if $t \neq \ell$ and $i \neq j$ then
\begin{eqnarray}
\nonumber
[I+Ae_{i,j}A^{-1},E_{t,\ell}(x)] &=&
(I+ x e_{t,\ell}+x b_{j,t}Ae_{i,\ell})(I-xe_{t,\ell}) \\
\nonumber
&=& I+xe_{t,\ell} + x b_{j,t} Ae_{i,\ell}-x e_{t,\ell} \\
\label{Eq:magic 3}
&=& I+x b_{j,t}Ae_{i,\ell}.
\end{eqnarray}
We are now ready to complete the proof of the lemma.
Choose $j,k$ such that $j \neq i$ and $k \neq \ell$.
Since $E_{i,j}(1)=I+e_{i,j}$,
\[
[A,E_{i,j}(1)]=AE_{i,j}(1)A^{-1}E_{i,j}(1)^{-1} =(I+Ae_{i,j}A^{-1})E_{i,j}(-1).
\]
Therefore,
\begin{multline}
\label{Eq:magic 4} 
[[ A,E_{i,j}(1)],E_{k,\ell}(x)] =
[(I+Ae_{i,j}A^{-1}) E_{i,j}(-1),E_{k,\ell}(x)] \\
=
{}^{(I+Ae_{i,j}A^{-1})}[E_{i,j}(-1),E_{k,\ell}(x)] 
[I+Ae_{i,j}A^{-1},E_{k,\ell}(x)].
\end{multline}
If $j \neq k$ then $[E_{i,j}(-1),E_{k,\ell}(x)]=I$ since $i \neq \ell$, so \eqref{Eq:magic 3} 
applied with $t=k$ shows that \eqref{Eq:magic 4} is equal to 
\[
I+xb_{j,k} A e_{i,\ell}
\]
as needed.
If $j=k$ then $[E_{i,j}(-1),E_{k,\ell}(x)]=E_{i,\ell}(-x)$.
Now \eqref{Eq:magic 2} applied with $t=i$ and \eqref{Eq:magic 3} applied with $t=k$, together with the fact that $e_{i,\ell}e_{i,\ell}=0$,
implies that \eqref{Eq:magic 4} is equal to
\[
(I-xe_{i,\ell}-xb_{j,i}Ae_{i,\ell})(I+xb_{j,j}Ae_{i,\ell}) =
I-xe_{i,\ell}-xb_{j,i}A e_{i,\ell} + x b_{j,j} A e_{i,\ell}.
\]
This completes the proof.
\end{proof}


\begin{definition}
An $n \times n$ matrix $H=(h_{i,j})$ over $\F R$ is called \emph{upper Hessenberg} if $h_{i,j}=0$ whenever $i>j+1$.
It is called \emph{lower Hessenberg} if $h_{i,j}=0$ whenever $j>i+1$.
\[
\OP{UP}=
\begin{pmatrix}
*      & * & \cdots & \cdots & * & *\\
*      & * & \cdots & \cdots & * & *\\
0      & * & \cdots & \cdots & * & *\\
0      & 0 & *      & \cdots & * & *\\
\vdots &   &        & \ddots &   & *\\
0      & 0 & \cdots & 0      & * & *
\end{pmatrix}
\qquad
\OP{LOW}=
\begin{pmatrix}
*      & * & 0 & 0      & \cdots & 0 \\
*      & * & * & 0      & \cdots & 0 \\
\vdots &   &   & \ddots &        & \vdots \\
*      & * &   &        &        & 0 \\
*      & * &   &        &        & * \\
 *      & * &   & \cdots&        & *
\end{pmatrix}
\]
\end{definition}

\begin{definition}\label{D:ESd}
Given a set of matrices $S \subseteq \SL(n,\F R)$ and $d \geq 0$, set
\[
\C E(S,d) = \{x \in \F R \ :\ E_{1,n}(x) \in B_S(d)\}.
\]
When $S=\{A\}$ for some $A \in \OP{SL}(n,\F R)$ we  write $\C E(A,d)$.
\end{definition}

\begin{remark}\label{R:easy facts on E}
With the notation of Definition~\ref{D:ESd}:
\begin{enumerate}[label=(\alph*)]
\item
If $A$ is conjugate to $B$ then clearly $\C E(A,d)=\C E(B,d)$.

\item
$x \in \C E(S,d) \implies -x \in \C E(S,d)$ because $E_{1,n}(x) \in B_S(d) \implies E_{1,n}(-x)=E_{1,n}(x)^{-1} \in B_S(d)$.

\item
Let $S^T= \{A^T \ :\ A\in S\}$, where $A^T$ denotes the transpose of $A$.
Then $\C E(S^T,d)=\C E(S,d)$ because $E_{n,1}(x)=E_{1,n}(x)^T$ and $\sigma_{1,n} E_{n,1}(x) \sigma_{1,n}^{-1}=E_{1,n}(-x)$.

\item
If $x \in \C E(S,d_1)$ and $y \in \C E(S,d_2)$ then $x+y \in \C E(S,d_1+d_2)$ because $E_{1,n}(x)E_{1,n}(y)=E_{1,n}(x+y)$.
\end{enumerate}
\end{remark}

%

\begin{lemma}
\label{L:pid eliminate one}
Let $R$ be a commutative ring with 1 and let $a,b,c \in R$.
Then $(ab-1)_R+(ac)_R=(ab-1)_R+(c)_R$.
\end{lemma}

\begin{proof}
Since $(ac)_R\subseteq (c)_R$ it suffices to show that $(c)_R\subseteq (ab-1)_R+(ac)_R$, which follows from $c= -c(ab-1) +b(ac)$.
\end{proof}

\begin{lemma}\label{L:ideals in E from Hessenberg}
Let $n \geq 3$ and let $A=(a_{i,j})$ be upper Hessenberg in $\SL(n,\F R)$.
Set $B=A^{-1}$ and write $B=(b_{i,j})$.  
Fix $1 \leq i,\ell \leq n$ such that $\ell>i+1$.
Then for any $j \neq i,\ell$,
\[
(( b_{j,j}-b_{j,i})a_{i,i}-1)_{\F R}+ \sum_{k \neq i}(a_{k,i})_{\F R} \subseteq \C E(A,4).
\]
\end{lemma}

\begin{proof}
Apply Lemma \ref{L:magic double commutators} to compute the double commutator:
\[
[[ A,E_{i,j}(1)],E_{j,\ell}(x)] = I+x(( b_{j,j}-b_{j,i})Ae_{i,\ell} - e_{i,\ell}).
\]
Since $A$ is upper Hessenberg and $i+1 <\ell$, the matrix on the
right-hand side has the form in 
Lemma~\ref{L:conjugate to elementary} (with $k=\ell$).
It is therefore conjugate to
$E_{1,n}(xt)$, where $t$ is the $\gcd$ of
\[
\{ (b_{j,j}-b_{j,i})a_{i,i}-1\} \cup \{ (b_{j,j}-b_{j,i})a_{k,i} : k \neq i\}.
\]
By Lemma~\ref{L:pid eliminate one},
\[
(t)_{\F R} = ((b_{j,j}-b_{j,i})a_{i,i}-1)_{\F R} + \sum_{k \neq i} (a_{k,i})_{\F R}.
\]
Since the double commutator above is in $B_A(4)$ and since $x \in \F R$ is arbitrary, it follows that $\C E(A,4) \supseteq \{ tx \ :\ x \in \F R\}=(t)_{\F R}$ as needed.
\end{proof}

We remark that if $u,v \in \F R^n$ are column vectors such that $u=Av$ for some $A \in \OP{GL}(n,\F R)$ then $\gcd(u_1,\dots,u_n)=\gcd(v_1,\dots,v_n)$ because every $u_i$ is a linear combination of $v_1,\dots,v_n$, and since $A$ is invertible, also every $v_i$ is a linear combination of $u_1,\dots,u_n$.

%
\begin{lemma}\label{L:conjugation to Hessenberg}
Let $n \geq 2$ and let $M=(m_{i,j})$  be an $n \times n$ matrix over $\F R$.
Then $M$ is conjugate by a matrix $A \in \SL(n,\F R)$ 
to an upper Hessenberg matrix $H=(h_{i,j})$ such that $h_{1,1}=m_{1,1}$ and
$h_{2,1}=\gcd(m_{2,1},m_{3,1},\dots,m_{n,1})$.
\end{lemma}

\begin{proof}
In the proof of \cite[Theorem III.1]{MR0340283} a matrix $A$ is constructed such that $H=A^{-1}MA$ is a lower Hessenberg matrix.
Moreover, $A$ is the product of matrices of determinant $1$ of the form $I+\alpha e_{p,p}+\beta e_{p,q} +\gamma e_{q,p}+\delta e_{q,q}$ for some $2 \leq p,q \leq n$.
Thus $A$ has the form $\left( \begin{smallmatrix} 1 & 0 \\ 0 & D \end{smallmatrix} \right)$.
Therefore $h_{1,1}=m_{1,1}$ and $(h_{1,2},0,\dots,0)=(m_{1,2},\dots,m_{1,n})\cdot D$, so $h_{1,2}=\gcd(m_{1,2},m_{1,3},\dots,m_{1,n})$.
By taking transposes, we obtain the statement of the lemma.
\end{proof}

%
%

\begin{lemma}\label{L:off diagonals in E}
Let $A \in \SL(n,\F R)$ where $n \geq 3$, and let $1\leq m\leq n$.
Then
\[
\sum_{k\neq m} (a_{k,m})_{\F R} \subseteq \C E(A,4).
\]
\end{lemma}

\begin{proof}
Let $B$ be the matrix obtained from $A$ by conjugating by the matrix $\sigma_{1,m}$ if $m >1$ and set $B=A$ if $m=1$.
The first column of $B$ is equal to the first column of $A$ if $m=1$, and if $m>1$ it is equal to
\[
(a_{m,m},a_{2,m},\dots,a_{m-1,m},-a_{1,m}, a_{m+1,m},\dots,a_{n,m}).
\]
Lemma~\ref{L:conjugation to Hessenberg} implies that $B$ is conjugate to an upper Hessenberg matrix $H$ whose first column is $(a_{m,m},t,0,\dots,0)$, where
\begin{eqnarray*}
t=\gcd(a_{2,1},\dots,\dots, a_{n,1}) & & \text{if $m=1$} \\
t=\gcd(a_{1,m},\dots, a_{m-1,m}, a_{m+1,m},\dots,a_{n,m}) & & \text{if $m>1$}.
\end{eqnarray*}
By Lemma~\ref{L:ideals in E from Hessenberg} applied to $H$ with $i=1,\ell=n$ and $j=2$, we see that $\C E(A,4)$ contains $(t)_{\F R}= \sum_{k \neq m} (a_{k,m})_{\F R}$.
\end{proof}

Let $R$ be any commutative ring with 1. 
Set
\[
\C M(R) =\{ \F m \lhd R \ :\ m \text{ is a maximal ideal}\}.
\]
Recall that 
\[
\OP{PSL}(n,R) = \OP{SL}(n,R) / \{ \text{the scalar matrices $\lambda I$, $\lambda \in R^\times$}\}.
\]

\begin{definition}\label{D:slnp-and-piA}
Let $\F R$ be a principal ideal domain and $I$ an ideal. 
Set
\[
\SL(n,I) \, \stackrel{\text{def}}{=} \, \OP{Ker}\big( \SL(n,\F R) \to \PSL(n,\F R/I) \big). 
\]
For any $A \in \SL(n,\F R)$ set
\[
\Pi(A) = \{ \F p \in \C M(\F R) \ :\ A \in \SL(n,\F p)\}.
\]
\end{definition}

\begin{remark}\label{R:PiA}
\begin{enumerate}[label=(\alph*)]
\item
Since $\F R$ is a p.i.d, it is a unique factorization domain, so every $0 \neq x \in \F R$ belongs to only finitely many prime ideals.  Hence $\Pi(A)$ is always finite, except if $A$ is a scalar matrix.  

\item
If $A$ and $B$ are conjugate matrices then $\Pi(A)=\Pi(B)$.

\item
\label{R:PiA:prod}
$\Pi(AB) \supseteq \Pi(A) \cap \Pi(B)$.
\end{enumerate}
\end{remark}

\begin{proposition}\label{P:prime ideals of A in E}
Let $\F R$ be a principal ideal domain and let $n \geq 3$.
Then for any $A \in \OP{SL}(n,\F R)$ there exists an ideal $I \lhd \F R$ contained in $\C E(A, 4n+4)$ such that for any $\F p \in \C M(\F R)$, 
\[
I \subseteq \F p \implies \F p \in \Pi(A).
\]
In fact, $I \subseteq J_1+\dots+J_{n+1}$ where $J_i \subseteq \C E(A,4)$ are ideals in $\F R$.
\end{proposition}

\begin{proof}
By Lemma \ref{L:conjugation to Hessenberg} and Remarks~\ref{R:easy facts on E} and \ref{R:PiA},  we may assume that $A$ is upper Hessenberg.  
Set $B= A^{-1}$ and let $a_{i,j}$ and $b_{i,j}$ denote the entries of $A$ and $B$.
From Lemma \ref{L:off diagonals in E} we obtain ideals
\begin{eqnarray}\label{Eq:primes AE 0}
&& J_1 = \sum_{k \neq n-1} (a_{k,n-1})_{\F R} \subseteq \C E(A,4) \\
\nonumber
&& J_2 = \sum_{k \neq n} (a_{k,n})_{\F R} \subseteq \C E(A,4).
\end{eqnarray}
By Remark~\ref{R:easy facts on E} and Lemma~\ref{L:ideals in E from Hessenberg} with $j= 1$ and $\ell= n$,
\begin{equation}\label{Eq:primes AE 1}
\sum_{i=2}^{n-2} \left(((b_{1,1}-b_{1,i})a_{i,i}-1)_{\F R}+ \sum_{k\neq i} (a_{k,i})_{\F R}\right) \subseteq \C E(A,4(n-3)).
\end{equation}
This is a sum of $n-3$ ideals $J_3,\dots, J_{n-1}$, each contained in $\C E(A,4)$.
Notice that this is the zero ideal if $n=3$. 
Applying Lemma~\ref{L:ideals in E from Hessenberg} with $i= 1$, $j= n-1$ and $\ell= n$, we obtain the ideal $J_n$ given by
\begin{equation}\label{Eq:primes AE 2}
((b_{n-1,n-1}-b_{n-1,1})a_{1,1}-1)_{\F R}+ \sum_{k\neq 1} (a_{k,1})_{\F R}  \subseteq \C E(A,4).
\end{equation}

Let $\F a$ be the ideal in $\F R$ generated by the off-diagonal elements of $A$:
\[
\F a = \sum_{i \neq j} (a_{i,j})_{\F R}.
\]
We claim that 
\[
b_{i,j} \in \F a \qquad \text{for all $i \neq j$.}
\]
Indeed, if $\F a \neq \F R$ then $A\!\! \mod \F a$ is a diagonal matrix in $\SL(n,\F R/\F a)$ and therefore so is $B\!\! \mod \F a = (A\!\! \mod \F a)^{-1}$, hence $B\!\! \mod \F a$ has its off-diagonal entries in $\F a$.

Assume $n \geq 4$.
We will show that the ideal 
\[
I= \F a + \sum_{i= 2}^{n- 2}(b_{1,1}a_{i,i}- 1)_{\F R}+ (b_{n-1, n-1}a_{1,1}- 1)_{\F R}+ (b_{n,n}a_{1,1}- 1)_{\F R} 
\]
has the properties stated in the proposition.

Suppose $I \subseteq \F p$ for some maximal ideal $\F p$ of $\F R$.
In particular, we must have $\F a \neq \F R$.
Set $\bar{A}=A\!\! \mod I$ and $\bar{B}=B\!\! \mod I$, and write $\bar{A}=(\bar{a}_{i,j})$ and $\bar{B}=(\bar{b}_{i,j})$.
Since $\F a \subseteq I$ it follows that $\bar{A}$ and $\bar{B}=\bar{A}^{-1}$ are diagonal and therefore $\bar{a}_{i,i}=\bar{b}_{i,i}^{-1}$ for all $1 \leq i \leq n$.
By the definition of $I$ we get that $\bar{a}_{1,1}=\bar{b}_{n-1,n-1}^{-1}$, $\bar{a}_{1,1}=\bar{b}_{n,n}^{-1}$ and $\bar{a}_{2,2},\dots,\bar{a}_{n-2,n-2}=\bar{b}_{1,1}^{-1}$.
Hence $\bar{a}_{1,1}=\bar{b}_{1,1}^{-1}=\bar{a}_{2,2}=\dots=\bar{a}_{n-2,n-2}$, $\bar{a}_{n,n}=\bar{b}_{n,n}^{-1}=\bar{a}_{1,1}$ and $\bar{a}_{n-1,n-1}=\bar{b}_{n-1,n-1}^{-1}=\bar{a}_{1,1}$.
We deduce that $\bar{a}_{1,1}=\dots=\bar{a}_{n,n}$.  It follows that $A\!\! \mod \F p$ is a scalar matrix, so $\F p \in \Pi(A)$.
It remains to show that $I \subseteq \C E(A,4n+4)$.

Applying Lemma \ref{L:ideals in E from Hessenberg} with $i= 1$, $j= n$ and $\ell= n-1$, we obtain an ideal $J_{n+1}$ given by
\begin{equation}\label{Eq:primes AE 3}
((b_{n,n}-b_{n,1})a_{1,1}-1)_{\F R}+ \sum_{k\neq 1} (a_{k,1})_{\F R}  \subseteq \C E(A,4).
\end{equation}
Let $J=\sum_{i=1}^{n+1}$ be the sum of the ideals from \eqref{Eq:primes AE 0}--\eqref{Eq:primes AE 3}.
By Remark \ref{R:easy facts on E},
\[
J \subseteq \C E(A,4n+4).
\]
It is clear that $\F a \subseteq J$.
Since $b_{i,j} \in \F a$ for all $i \neq j$ it follows that $b_{1,1}a_{i,i}-1 \in J$ for all $2 \leq i \leq n-2$ (by \eqref{Eq:primes AE 1}), that $b_{n-1,n-1}a_{1,1}-1 \in J$ (by \eqref{Eq:primes AE 2}) and that $b_{n,n}a_{1,1}-1 \in J$ (by \eqref{Eq:primes AE 3}).
We deduce that $I \subseteq J$ (in fact equality holds) and 
this proves the proposition for $n \geq 4$.

Assume that $n=3$.
The argument above does not go through since Lemma \ref{L:ideals in E from Hessenberg} cannot be applied to deduce \eqref{Eq:primes AE 3}.
Define 
\[
I= \F a + (b_{2, 2}a_{1,1}- 1)_{\F R}+ (b_{1,1}a_{3,3}- 1)_{\F R}.
\]
We will show that $I$ has the required properties.
Let $\F p$ be a maximal ideal containing $I$.
Thus $\F a \neq \F R$.
Let $\bar{A}=A\!\! \mod I$ and $\bar{B} = B\!\! \mod I$ as above.
Since $\F a \subseteq I$ these matrices are diagonal and $\bar{a}_{i,i}=\bar{b}_{i,i}^{-1}$ for all $1 \leq i \leq 3$.
By the definition of $I$ also, $\bar{a}_{1,1}=\bar{b}_{2,2}^{-1}=\bar{a}_{2,2}$ and $\bar{a}_{3,3}=\bar{b}_{1,1}^{-1}=\bar{a}_{1,1}$, so $\bar{A}$ is a scalar matrix.
Therefore $A\!\! \mod \F p$ is a scalar matrix so $\F p \in \Pi(A)$, as needed.
It remains to show that $I \subseteq \C E(A,16)$.

Set $M=\left(\begin{smallmatrix}0 & 0 & 1 \\ 0 & -1 & 0 \\ 1 & 0 & 0 \end{smallmatrix}\right) \in \SL(3,\F R)$.
Set $C=(MAM^{-1})^T$ and $D=(MBM^{-1})^T$.
Then 
\[
C=\left( \begin{smallmatrix} a_{3,3} & -a_{2,3} & a_{1,3} \\ -a_{3,2} & a_{2,2} & -a_{1,2} \\ 0 & -a_{2,1} & a_{1,1} \end{smallmatrix} \right),
D=\left( \begin{smallmatrix} b_{3,3} & -b_{2,3} & b_{1,3} \\ -b_{3,2} & b_{2,2} & -b_{1,2} \\ b_{3,1} & -b_{2,1} & b_{1,1} \end{smallmatrix} \right).
\]
Applying Lemma~\ref{L:ideals in E from Hessenberg} to the Hessenberg matrix $C$ with $i= 1$, $j= 2$ and $\ell= 3$ and using Remark \ref{R:easy facts on E}, we see that
\begin{equation}\label{Eq:primes AE 3-3}
((b_{2,2}+b_{3,2})a_{3,3}-1)_{\F R} 
\subseteq \C E(C,4) = \C E(A,4).
\end{equation}
Let $J$ be the sum of the ideals in \eqref{Eq:primes AE 0}, \eqref{Eq:primes AE 2} and \eqref{Eq:primes AE 3-3}.
Then $J \subseteq \C E(A,8+ 4+4)= \C E(A,16)$ by Remark~\ref{R:easy facts on E}.
It is easy to check that $\F a \subseteq J$.
Since $b_{3,2}, b_{2,1} \in \F a$ it follows that $b_{2,2}a_{3,3}-1$ and  $b_{2,2}a_{1,1}-1 \in J$.
Therefore $I \subseteq J$, 
and this completes the proof.
\end{proof}

\begin{corollary}\label{C:empty intersection of Pi}
Let $\F R$ be a principal ideal domain and $n \geq 3$.
Let $S=\{ A_1,\dots,A_k \} \subseteq \OP{SL}(n,\F R)$.
Then there are ideals $J_1,\dots,J_{k(n+1)}$, each contained in $\C E(S,4)$, and an ideal $I \subseteq \sum_{i} J_i$ such that for any $\F p \in \C M(\F R)$, if $\F p \supseteq I$ then $p \in \bigcap_{A \in S} \Pi(A)$.
\end{corollary}

\begin{proof}
For each $A_i$ choose $I_i \subseteq J_{i,1} + \dots+J_{i,n+1}$ as in Proposition \ref{P:prime ideals of A in E} and set $I=\sum_i I_i$.
\end{proof}

\begin{proposition}\label{P:S conjugation generates:revised}
Let $\F R$ be a p.i.d., let $n \geq 3$ and let $S$ be a finite subset of $\OP{SL}(n,\F R)$.
Then $\langle \langle S \rangle \rangle \supseteq  \OP{EL}(n,\F R)$ if and only if $\bigcap_{A \in S} \Pi(A) = \emptyset$.
In this case there are ideals $J_1,\dots,J_{k(n+1)} \subseteq \C E(S,4)$ such that $J_1+\dots+J_{k(n+1)}=\F R$.
\end{proposition}

\begin{proof}   
Suppose that $\F p \in \bigcap_{A \in S} \Pi(A)$.
Then $S$ is contained in the normal subgroup $\SL(n,\F p)$ (see Definition \ref{D:slnp-and-piA}) and in particular $E_{1,n}(1) \notin \langle \langle S \rangle \rangle$.
Conversely, suppose that $\bigcap_{A \in S} \Pi(A)=\emptyset$.
Let $I$ and $J_1,\dots,J_{k(n+1)}$ be as in Corollary \ref{C:empty intersection of Pi}.
Then $I$ is not contained in any $\F p \in \C M(\F R)$, so $I=\F R$.  Hence $E_{1,n}(1)\in \langle \langle S \rangle \rangle$.  Lemma~\ref{L:conjugate elementary matrices}(\ref{L:conjugate elementary matrices:2}) implies that $\langle \langle S \rangle \rangle \supseteq  \OP{EL}(n,\F R)$.
\end{proof}

\begin{proof}[Proof of Theorem \ref{T:slnr main theorem}]
Suppose that $S$ normally generates $\SL(n,\F R)$ and $|S|=k$.
Then $\bigcap_{A \in S} \Pi(A) = \emptyset$ by Proposition \ref{P:S conjugation generates:revised} and $\F R \subseteq \C E(A,4k(n+1))$ by Remark \ref{R:easy facts on E}, so $\OP{EL}(n,\F R)\subseteq B_A(4k(n+1))$ by Lemma~\ref{L:conjugate elementary matrices}(\ref{L:conjugate elementary matrices:1}).
Lemma \ref{L:BXn properties}\ref{L:BXn:assoc} shows that $\|\SL(n,\B R)\|_S \leq 4k(n+1)C_n$ and since $S$ was arbitrary, this proves the inequality on the right.

To prove the inequality on the left choose some $k \geq 1$.
Let $\F p_1 , \F p_2, \ldots ,\F p_k$ be distinct maximal ideals generated by $p_1,p_2,\dots,p_k \in \F R$.
For any $1 \leq i \leq k$ let $r_i$ be the product of all the elements $p_j$ except $p_i$.
For any $1 \leq i \leq k$ set $A_i=E_{1,n}(r_i)$.
Then
\[
\Pi(A_i)= \{ \F p_j \ :\ j \neq i\}.
\]
By Proposition \ref{P:S conjugation generates:revised} and the hypothesis on $\OP{EL}(n,\F R)$ the set $S=\{ A_1,\dots,A_k\}$ normally generates $\OP{SL}(n,\F R)$.
It remains to show that $\| \OP{SL}(n,\F R)\|_S \geq k$, or equivalently that $B_S(k-1) \subsetneq \SL(n,\F R)$.
We will show that $E_{1n}(1) \notin B_S(k-1)$.
If $X \in B_S(k-1)$ then $X=X_1\cdots X_{k-1}$ where each $X_i$ is conjugate to an element of $S$ or its inverse, hence $\Pi(X_i)$ is a subset of size $k-1$ of $\{ \F p_1,\dots,\F p_k\}$.
Therefore $\bigcap_{i=1}^{k-1} \Pi(X_i)$ is not empty and by Remark \ref{R:PiA}\ref{R:PiA:prod} $\F p_i \in \Pi(X)$ for some $i$, so in particular $X \neq E_{1n}(1)$.
\end{proof}

\begin{remarknonum}
It is instructive to show directly that $G=\OP{SL}(n,\B Z)$ is not uniformly bounded by following the last part of the proof.
Given primes $p_1,\dots,p_k$, set $r_i=p_1 \dots \widehat{p_i} \dots p_k$ and $g_i=E_{1,n}(r_i)$.
By the Chinese remainder theorem $E_{1,n}(1)=g_1^{f_1} \dots g_k^{f_k}$ for some $f_1,\dots,f_k$ so Lemma \ref{L:conjugate elementary matrices}(\ref{L:conjugate elementary matrices:2}) and Carter-Keller's result \cite{MR704220} show that $S=\{ g_1,\dots,g_k\}$ normally generates $G$.
However, any product of $k-1$ elements of $S$ must be contained in the congruence subgroup $\SL(n,p_i\B Z)$ for some $i$ so cannot be $E_{1,n}(1)$ and in particular  $\| G\|_S \geq k$.
\end{remarknonum}



\begin{proposition}\label{P:bound SLnR after Carter-Keller}
Let $\F R$ be a p.i.d and
let $m \geq 2$.
If $\OP{EL}(m,\F R)$ normally generates $\SL(m,\F R)$  and  $\| \SL(m,\F R)\|_{\OP{EL}(m,\F R)} \leq C$ then $\SL(n,\F R)$ is normally generated by $\OP{EL}(n,\F R)$ for any $n \geq m$ and $\| \SL(n,\F R) \|_{\OP{EL}(n,\F R)} \leq C+4(n-m)$.
\end{proposition}

\begin{proof}  
Use induction on $n$.
The base $n=m$ of the induction is trivial.
We carry out the induction step for $n >m$.
We will write $S=\OP{EL}(n,\F R)$ for short and notice that $n \geq 3$.

Let $A \in \SL(n,\F R)$.
Suppose first that $A$ is a block matrix $\left(\begin{smallmatrix} 1 & y \\ 0 & B \end{smallmatrix}\right)$ where $B \in \SL(n-1,\F R)$.
Then $A=\left(\begin{smallmatrix} 1 & 0 \\ 0 & B \end{smallmatrix}\right) \cdot \left(\begin{smallmatrix} 1 & y \\ 0 & I \end{smallmatrix}\right)$ so $A \in B_S(C+4(n-1-m)+1)$ by the induction hypothesis, Lemma \ref{L:conjugate to elementary} and Lemma \ref{L:BXn properties}\ref{L:BXn:mult}.

Now consider any $A \in \SL(n,\F R)$.
By Lemma \ref{L:conjugation to Hessenberg} we may assume that $A$ is upper Hessenberg.
Say its first column is $(a,b,0,\dots,0)$.
Then $\gcd(a,b)=1$ so $sa+tb=1$ for some $s,t \in \F R$ and one checks that
\[
(I-b e_{2,1} - e_{3,1}) \cdot (I+ (1-a)e_{1,3}) \cdot  (I+s e_{3,1} + t e_{3,2}) \cdot A 
\]
has the form $\left(\begin{smallmatrix} 1 & y \\ 0 & B \end{smallmatrix}\right)$.
It folows that from Lemmas \ref{L:conjugate to elementary} and \ref{L:BXn properties}\ref{L:BXn:mult} that $A \in B_S(C+4(n-1-m)+1+3) = B_S(C+4(n-m))$ and the induction step is complete.
\end{proof}

\begin{proof}[Proof of Corollary \ref{C:Carter Keller implications}]
By \cite{MR704220} 
$\OP{EL}(3,\C O)$ normally generates $\SL(3,\C O)$  and $\|\SL(3,\C O)\|_{\OP{EL}(3,\C O)} \leq 63$.
Also, $\C O$ is a principal ideal domain since the class number of its number field is one, and it has infinitely many maximal ideals since $\B Z \subseteq \C O$ and every maximal ideal of $\B Z$ extends to a maximal ideal of $\C O$.

Apply Proposition~\ref{P:bound SLnR after Carter-Keller} with $m=3$ to deduce that $\SL(n,\C O)$ is normally generated by $\OP{EL}(n,\C O)$ for any $n \geq 3$  and that
$
C_n=\|\SL(n,\C O)\|_{\OP{EL}(n,\C O)} \leq 63+4(n-3)=4n+51
$.
The result follows from Theorem \ref{T:slnr main theorem}.
\end{proof}

We now consider p.i.d's $\F R$ with only finitely many maximal ideals.

\begin{lemma}\label{L:units in finitely many max ideals}
Let $\F R$ be a principal ideal domain which has only finitely many maximal ideals.
Then for any $a, b \in \F R$ such that $\gcd(a,b)=1$, there exists $x \in \F R$ such that $a+bx$ is a unit in $\F R$.
\end{lemma}

\begin{proof}
Let $p_1,\dots,p_k$ be generators of the maximal ideals $\F p_1,\dots,\F p_k$ of $\F R$.
For any $r \in \F R$ set $\pi(r)=\{i \ :\ r \in \F p_i\}$.
Since $\gcd(a,b)= 1$ it follows that $\pi(a) \cap \pi(b)=\emptyset$.
Set $x=\prod_{i \notin \pi(a)} p_i$.
Then clearly $\pi(xb)=\pi(x)$ and for any $1 \leq i \leq k$ we have $a \in \F p_i \iff bx \notin \F p_i$, hence $a+bx \notin \F p_i$.
Since $a+bx$ does not belong to any maximal ideal it is invertible.
%
\end{proof}

\begin{proposition}\label{P:slnr bounded r local}
Let $\F R$ be a principal ideal domain with only finitely many maximal ideals.
Let $n \geq 1$.
Then $\OP{EL}(n,\F R)$ normally generates $\OP{SL}(n,\F R)$  and $\| \OP{SL}(n,\F R) \|_{\OP{EL}(n,\F R)} \leq 3(n-1)$.
\end{proposition}

\begin{proof} 
We use induction on $n$; the case $n=1$ is a triviality since $\SL(1,\F R)$ is trivial (and $\OP{EL}(1,\F R)=\emptyset$).
Assume that $n \geq 2$ and let $A \in \SL(n,\F R)$.
It acts in the standard way on the set of column vectors $\F R^n$ with the standard basis $e_1,\dots,e_n$.
By Lemma \ref{L:conjugation to Hessenberg} we may assume that $A$ is upper Hessenberg.
By Lemma \ref{L:units in finitely many max ideals} there is some $x \in \F R$ such that $B=E_{2,1}(x) \cdot A$ is upper Hessenberg with the entry $b_{21}$ a unit.
Set $U=E_{1,2}(b_{21}^{-1}(b_{11}-1))$ and $C=UBU^{-1}$.
By inspection the first column of $C$ is $(1,b_{21},0,\dots,0)$.
Then $D=E_{2,1}(-b_{21}) \cdot C$ is a block matrix $\left(\begin{smallmatrix} 1 & * \\ 0 & Q \end{smallmatrix}\right)$ where $Q \in \SL(n-1,\F R)$. 
Let $F=\left(\begin{smallmatrix} 1 & 0 \\ 0 & Q \end{smallmatrix}\right)$.
Then
$D \cdot F^{-1}$ is conjugate to an elementary matrix by Lemma \ref{L:conjugate to elementary}.
By applying the induction hypothesis to $Q$ it follows that $\|A\|_{\OP{EL}(n,\F R)} \leq \|F\|_{\OP{EL}(n,\F R)}+3 \leq 3(n-2)+3$ and the induction step follows.
\end{proof}

\begin{proof}[Proof of Theorem \ref{T:slnr main theorem:2}]
By Proposition \ref{P:slnr bounded r local}, $\OP{SL}(n,\F R)$ is normally generated by $\OP{EL}(n,\F R)$ and $\| \OP{SL}(n,\F R) \|_{\OP{EL}(n,\F R)} \leq 3(n-1)$.
Let $S \subseteq \SL(n,\F R)$ be normally generating, $|S|=k$.
By Proposition \ref{P:S conjugation generates:revised}, $\F R=J_1+\dots+J_{k(n+1)}$, a sum of ideals in $\C E(S,4)$.
For any $\F p \in \C M(\F R)$ there must exist $J_i$ such that $J_i \nsubseteq \F p$.
Thus, if $d \leq k(n+1)$ then $\F R= J_{i_1}+\dots+J_{i_d} \subseteq \C E(S,4d)$ since $|\C M(\F R)|=d$.
We deduce that $\F R = \C E(S,4 \min\{d,k(n+1)\})$, so $\|\OP{EL}(n,\F R)\|_S \leq 4 \min\{ d, k(n+1)\}$ by Lemma~\ref{L:conjugate elementary matrices}(\ref{L:conjugate elementary matrices:1}).
Then $\|\OP{SL}(n,\F R)\|_S \leq 12(n-1) \cdot \min\{ d,k(n+1)\}$ by Lemma \ref{L:BXn properties}\ref{L:BXn:assoc}.
\end{proof}

\begin{remark}\label{R:E does not generate SLn}
Let $\F R$ be a p.i.d.
In general $\OP{EL}(n,\F R)$ need not normally generate $\SL(n,\F R)$.
To see this, recall that $\OP{SK}_1(\F R)$ (in the sense of algebraic $K$-theory) is the group $\OP{SL}(\F R)/\OP{E}(\F R)$ where $\OP{SL}(\F R)=\bigcup_{n \geq 1} \OP{SL}(n,\F R)$ and $\OP{E}(\F R)$ is the subgroup normally generated by $\bigcup_{n \geq 1}\OP{EL}(n,\F R)$.

Let $\F R$ be the ring $\B Z[T]$ with the polynomials $T$ and $T^m-1$ inverted for all $m \geq 1$.  
This is a p.i.d by \cite{MR601681}.
Also, \cite[Proposition 8]{MR601681} shows that $\OP{SK}_1(\F R) \neq 0$ and therefore $\OP{SL}(n,\F R)$ is not normally generated by $\OP{EL}(n,\F R)$ for all sufficiently large $n$.  
It follows from Proposition~\ref{P:bound SLnR after Carter-Keller} that $\OP{EL}(n,\F R)$ does not normally generate $\OP{SL}(n,\F R)$ for any $n \geq 2$.  
\end{remark}

\begin{remark}
After this paper was submitted, Trost has proved generalisations of Theorem~\ref{T:slnr main theorem} for a wider class of split semisimple Chevalley groups $G$ \cite{trost1,trost2,trost3,trost4}, and has made some progress towards proving analogous results for finite index subgroups of $G$ \cite{trost3}, \cite[Remark 3.8]{trost1}.  He has also established lower bounds involving $n$ for $\Delta_k(G)$ where $G$ is symplectic \cite[Theorem~2]{trost2}; it seems likely that similar bounds can be established for other Dynkin types.
\end{remark}

\section{Applications to finite groups}\label{S:applications to finite groups}

We write $\log x$ to mean $\log_2 x$.

\begin{proposition}\label{P:minimal ccs}
Let $G$ be a finite group with $|G|>3$ and let $S$ be the conjugacy class of some $s \in G$.
If $s$ normally generates $G$ then 
\[
 \log |S| > \frac{\log |G|}{\Delta(G)}- 2.
\]
If $s^{-1} \in S$ then $\log |S| > \frac{\log |G|}{\Delta(G)}- 1$.
\end{proposition}

\begin{proof}
Clearly, $B_S(1)=\{1\} \cup S \cup S^{-1}$, and $B_S(1) \subsetneq G$ since $|G|> 3$.  
Hence $\Delta(G)> 1$.  
Since $|G|_S \leq \Delta(G)$, we obtain a surjective function $B_S(1)^{\Delta(G)} \to G$.
The preimage of $1 \in G$ contains at least $2$ elements, so $(1+|S \cup S^{-1}|)^{\Delta(G)} > |G|$.
This gives $2|S \cup S^{-1}| > |G|^{1/\Delta(G)}$ and the result follows.
\end{proof}

\begin{example}\label{E:PSL(n,q) ccs}
Let $G=\PSL(n,\B F_q)$ where $n \geq 3$.
By Theorem \ref{T:slnr main theorem:2} and Lemma \ref{L:finite gen quotient not strongly bounded}\ref{I:finite gen quotient not strongly bounded:2}, $\Delta(G) \leq 12(n-1)$, and the simplicity of $G$ implies that any non-trivial $s \in G$ normally generates.
Proposition \ref{P:minimal ccs} shows that if $S \subseteq G$ is any non-trivial conjugacy class then
\[
\log |S| > \frac{\log |G|}{ \Delta(G)}- 2 \geq
\frac{\log(\frac{q^{n(n-1)/2}(q^n-1)(q^{n-1}-1)\dots(q^2-1)}{\gcd(n, q-1)})}{12(n-1)}- 2.
\]
For any $x \geq 2$ we have $\log(x-1) \geq \log x-\frac{c}{x-1}\geq \log x- \frac{2c}{x}$ where $c= \ln(2)^{-1}$. 
Since $\frac{2c}{q^2}+\frac{2c}{q^3}+\dots \leq 2$ and since $q \geq \gcd(n,q-1)$, we may continue the inequality above
\begin{multline*}
\geq  
\frac{\frac{1}{2}n(n-1)\log q + \sum_{k=2}^n (k \log q -\frac{2c}{q^k}) - \log q}{12(n-1)} -2  \\
=
\frac{\log q}{12}(n+1) - \frac{2+\log q}{12(n-1)} -2.
\end{multline*}
\end{example}

\begin{proof}[Proof of Proposition \ref{P:slnzl ccs}]

The localised ring $\F R:= \B Z_{(p_1,\dots,p_k)}$ has exactly $k$ prime ideals generated by $p_1,\dots,p_k \in \B Z$.
We claim that reduction modulo $\ell$ gives rise to an epimorphism $\pi \colon \OP{SL}(n,\F R) \to \OP{SL}(n,\B Z/\ell)$.
To see this, for any $A \in \OP{SL}(n,\B Z/\ell)$ choose a matrix $\hat{A}$ with entries in $\B Z$ such that $\hat{A}\!\! \mod \ell =A$.
Set $u=\det(\hat{A})$.
Then $u=1\!\! \mod \ell$ so $u \in \F R^\times$.
Therefore $\hat{A} \cdot \OP{diag}(u^{-1},1,\dots,1) \in \OP{SL}(n, \F R)$  is a preimage of $A$.
Lemma \ref{L:finite gen quotient not strongly bounded}\ref{I:finite gen quotient not strongly bounded:2} and Theorem \ref{T:slnr main theorem:2} give $\Delta(\OP{SL}(n,\B Z/\ell)) \leq 12k(n-1)$.

Propositions \ref{P:slnr bounded r local} and  \ref{P:S conjugation generates:revised} imply that a matrix $A \in \SL(n,\ZZ/\ell)$ with $A \mod p_i$ non-scalar normally generates $\OP{SL}(n,\B Z/\ell)$.
Proposition \ref{P:minimal ccs} gives the lower bound on the size of the conjugacy class of $A$.
\end{proof}

\bibliography{bibliography}

\def\polhk#1{\setbox0=\hbox{#1}{\ooalign{\hidewidth
  \lower1.5ex\hbox{`}\hidewidth\crcr\unhbox0}}}
  \def\polhk#1{\setbox0=\hbox{#1}{\ooalign{\hidewidth
  \lower1.5ex\hbox{`}\hidewidth\crcr\unhbox0}}}
  \def\polhk#1{\setbox0=\hbox{#1}{\ooalign{\hidewidth
  \lower1.5ex\hbox{`}\hidewidth\crcr\unhbox0}}}
  \def\polhk#1{\setbox0=\hbox{#1}{\ooalign{\hidewidth
  \lower1.5ex\hbox{`}\hidewidth\crcr\unhbox0}}}
  \def\polhk#1{\setbox0=\hbox{#1}{\ooalign{\hidewidth
  \lower1.5ex\hbox{`}\hidewidth\crcr\unhbox0}}}
  \def\polhk#1{\setbox0=\hbox{#1}{\ooalign{\hidewidth
  \lower1.5ex\hbox{`}\hidewidth\crcr\unhbox0}}} \def\cprime{$'$}
\begin{thebibliography}{10}

\bibitem{MR1612569}
Vladimir~I. Arnold and Boris~A. Khesin.
\newblock {\em Topological methods in hydrodynamics}, volume 125 of {\em
  Applied Mathematical Sciences}.
\newblock Springer-Verlag, New York, 1998.

\bibitem{MR3063901}
Tim Austin and Calvin~C. Moore.
\newblock Continuity properties of measurable group cohomology.
\newblock {\em Math. Ann.}, 356(3):885--937, 2013.

\bibitem{MR1445290}
Augustin Banyaga.
\newblock {\em The structure of classical diffeomorphism groups}, volume 400 of
  {\em Mathematics and its Applications}.
\newblock Kluwer Academic Publishers Group, Dordrecht, 1997.

\bibitem{MR543215}
Philippe Blanc.
\newblock Sur la cohomologie continue des groupes localement compacts.
\newblock {\em Ann. Sci. \'Ecole Norm. Sup. (4)}, 12(2):137--168, 1979.

\bibitem{MR1102012}
Armand Borel.
\newblock {\em Linear algebraic groups}, volume 126 of {\em Graduate Texts in
  Mathematics}.
\newblock Springer-Verlag, New York, second edition, 1991.

\bibitem{MR2109105}
Nicolas Bourbaki.
\newblock {\em Lie groups and {L}ie algebras. {C}hapters 7--9}.
\newblock Elements of Mathematics (Berlin). Springer-Verlag, Berlin, 2005.
\newblock Translated from the 1975 and 1982 French originals by Andrew
  Pressley.

\bibitem{MR3136522}
Daniel Bump.
\newblock {\em Lie groups}, volume 225 of {\em Graduate Texts in Mathematics}.
\newblock Springer, New York, second edition, 2013.

\bibitem{MR2509711}
Dmitri Burago, Sergei Ivanov, and Leonid Polterovich.
\newblock Conjugation-invariant norms on groups of geometric origin.
\newblock In {\em Groups of diffeomorphisms}, volume~52 of {\em Adv. Stud. Pure
  Math.}, pages 221--250. Math. Soc. Japan, Tokyo, 2008.

\bibitem{MR2680425}
Marc Burger, Alessandra Iozzi, and Anna Wienhard.
\newblock Surface group representations with maximal {T}oledo invariant.
\newblock {\em Ann. of Math. (2)}, 172(1):517--566, 2010.

\bibitem{MR3198721}
Shawn~T. Burkett and Hung~Ngoc Nguyen.
\newblock Conjugacy classes of small sizes in the linear and unitary groups.
\newblock {\em J. Group Theory}, 16(6):851--874, 2013.

\bibitem{MR704220}
David Carter and Gordon Keller.
\newblock Bounded elementary generation of {${\rm SL}_{n}({\mathcal O})$}.
\newblock {\em Amer. J. Math.}, 105(3):673--687, 1983.

\bibitem{MR2854105}
Indira Chatterji, Guido Mislin, Christophe Pittet, and Laurent Saloff-Coste.
\newblock A geometric criterion for the boundedness of characteristic classes.
\newblock {\em Math. Ann.}, 351(3):541--569, 2011.

\bibitem{MR2240370}
Yves de~Cornulier.
\newblock Strongly bounded groups and infinite powers of finite groups.
\newblock {\em Comm. Algebra}, 34(7):2337--2345, 2006.

\bibitem{delzant}
Thomas Delzant.
\newblock Sous-alg\`ebres de dimension finie de l'alg\`ebre des champs
  hamiltoniens.
\newblock Available at: {\tt
  http://www-irma.u-strasbg.fr/\~\/delzant/preprint.html}, 1995.

\bibitem{MR3920346}
Philip~A. Dowerk and Andreas Thom.
\newblock Bounded normal generation for projective unitary groups of certain
  infinite operator algebras.
\newblock {\em Int. Math. Res. Not. IMRN}, (24):7642--7654, 2018.

\bibitem{MR3907832}
Philip~A. Dowerk and Andreas Thom.
\newblock Bounded normal generation and invariant automatic continuity.
\newblock {\em Adv. Math.}, 346:124--169, 2019.

\bibitem{MR2819193}
{\'S}wiatos{\l}aw~R. Gal and Jarek K\k{e}dra.
\newblock On bi-invariant word metrics.
\newblock {\em J. Topol. Anal.}, 3(2):161--175, 2011.

\bibitem{MR3085032}
Jakub Gismatullin.
\newblock Boundedly simple groups of automorphisms of trees.
\newblock {\em J. Algebra}, 392:226--243, 2013.

\bibitem{MR601681}
Daniel~R. Grayson.
\newblock {$SK_{1}$}\ of an interesting principal ideal domain.
\newblock {\em J. Pure Appl. Algebra}, 20(2):157--163, 1981.

\bibitem{MR3025417}
Joachim Hilgert and Karl-Hermann Neeb.
\newblock {\em Structure and geometry of {L}ie groups}.
\newblock Springer Monographs in Mathematics. Springer, New York, 2012.

\bibitem{MR1336822}
Morris~W. Hirsch.
\newblock {\em Differential topology}, volume~33 of {\em Graduate Texts in
  Mathematics}.
\newblock Springer-Verlag, New York, 1994.
\newblock Corrected reprint of the 1976 original.

\bibitem{MR0396773}
James~E. Humphreys.
\newblock {\em Linear algebraic groups, corrected fifth printing}.
\newblock Springer-Verlag, New York-Heidelberg, 1975.
\newblock Graduate Texts in Mathematics, No. 21.

\bibitem{MR1920389}
Anthony~W. Knapp.
\newblock {\em Lie groups beyond an introduction}, volume 140 of {\em Progress
  in Mathematics}.
\newblock Birkh\"auser Boston, Inc., Boston, MA, second edition, 2002.

\bibitem{MR364552}
Bertram Kostant.
\newblock On convexity, the {W}eyl group and the {I}wasawa decomposition.
\newblock {\em Ann. Sci. \'{E}cole Norm. Sup. (4)}, 6:413--455 (1974), 1973.

\bibitem{MR96h:58063}
Fran{\c{c}}ois Lalonde and Dusa McDuff.
\newblock Hofer's {$L\sp \infty$}-geometry: energy and stability of
  {H}amiltonian flows. {I}, {II}.
\newblock {\em Invent. Math.}, 122(1):1--33, 35--69, 1995.

\bibitem{leroux-mann}
Fr\'ed\'eric Le~Roux and Kathryn Mann.
\newblock Strong distortion in transformation groups.
\newblock Available at
  \url{http://www.math.brown.edu/~mann/papers/distortion.pdf}, 2016.

\bibitem{MR1308046}
Alexander Lubotzky.
\newblock {\em Discrete groups, expanding graphs and invariant measures},
  volume 125 of {\em Progress in Mathematics}.
\newblock Birkh\"auser Verlag, Basel, 1994.
\newblock With an appendix by Jonathan D. Rogawski.

\bibitem{MR1978431}
Alexander Lubotzky and Dan Segal.
\newblock {\em Subgroup growth}, volume 212 of {\em Progress in Mathematics}.
\newblock Birkh\"auser Verlag, Basel, 2003.

\bibitem{MR0089998}
George~W. Mackey.
\newblock Les ensembles bor\'eliens et les extensions des groupes.
\newblock {\em J. Math. Pures Appl. (9)}, 36:171--178, 1957.

\bibitem{MR2000g:53098}
Dusa McDuff and Dietmar Salamon.
\newblock {\em Introduction to symplectic topology}.
\newblock Oxford Mathematical Monographs. The Clarendon Press Oxford University
  Press, New York, second edition, 1998.

\bibitem{MR1840942}
Nicolas Monod.
\newblock {\em Continuous bounded cohomology of locally compact groups}, volume
  1758 of {\em Lecture Notes in Mathematics}.
\newblock Springer-Verlag, Berlin, 2001.

\bibitem{MR0171880}
Calvin~C. Moore.
\newblock Extensions and low dimensional cohomology theory of locally compact
  groups. {I}, {II}.
\newblock {\em Trans. Amer. Math. Soc.}, 113:40--63, 1964.

\bibitem{MR2357719}
Dave~Witte Morris.
\newblock Bounded generation of {${\rm SL}(n,A)$} (after {D}. {C}arter, {G}.
  {K}eller, and {E}. {P}aige).
\newblock {\em New York J. Math.}, 13:383--421, 2007.

\bibitem{MR0340283}
Morris Newman.
\newblock {\em Integral matrices}.
\newblock Academic Press, New York-London, 1972.
\newblock Pure and Applied Mathematics, Vol. 45.

\bibitem{0803.4165v5}
Nikolay Nikolov.
\newblock Strong approximation methods in group theory, an {LMS}/{EPSRC} short
  course lecture notes.
\newblock Available at \url{http://arxiv.org/abs/0803.4165v5}, 2008.

\bibitem{MR3241729}
Leonid Polterovich and Daniel Rosen.
\newblock {\em Function theory on symplectic manifolds}, volume~34 of {\em CRM
  Monograph Series}.
\newblock American Mathematical Society, Providence, RI, 2014.

\bibitem{MR0498810}
John~S. Rose.
\newblock {\em A course on group theory}.
\newblock Cambridge University Press, Cambridge-New York-Melbourne, 1978.

\bibitem{MR0494071}
James~D. Stasheff.
\newblock Continuous cohomology of groups and classifying spaces.
\newblock {\em Bull. Amer. Math. Soc.}, 84(4):513--530, 1978.

\bibitem{trost2}
Alexander Trost.
\newblock Explicit strong boundedness for higher rank symplectic groups.
\newblock Dec 2020.
\newblock Available at \url{https://arxiv.org/abs/2020.12328}.

\bibitem{trost1}
Alexander Trost.
\newblock Strong boundedness of simply connected split chevalley groups defined
  over rings.
\newblock {\em Israel Journal of Mathematics, to appear}, Apr 2020.
\newblock Available at \url{https://arxiv.org/abs/2004.05039}.

\bibitem{trost3}
Alexander Trost.
\newblock Bounded generation for congruence subgroups of sp4(r).
\newblock Jan 2021.
\newblock Available at \url{https://arxiv.org/abs/2101.02301}.

\bibitem{trost4}
Alexander Trost.
\newblock Strong boundedness of sl2(r) for rings of s-algebraic integers with
  infinitely many units.
\newblock May 2021.
\newblock Available at \url{https://arxiv.org/abs/2105.10972}.

\end{thebibliography}
\bibliographystyle{plain}

\end{document}